\documentclass[reqno,11pt,a4paper]{amsart}
\usepackage{graphicx}
\newtheorem{theorem}{Theorem}
\newtheorem{lemma}[theorem]{Lemma}
\newtheorem*{lemma*}{Lemma}
\newtheorem{proposition}[theorem]{Proposition}

\usepackage[utf8]{inputenc}
\usepackage[english]{babel}
\usepackage{latexsym}
\usepackage{amssymb}
\usepackage{amsmath}

  \newtheorem{remark}[theorem]{Remark}

\theoremstyle{remark}

\renewcommand{\Im}{\mathrm{Im}\,}

\newcommand\otimesal{\mathop{\hbox{\raise 1.6 ex
  \hbox{$\scriptscriptstyle\mathrm{al}$}
\kern -0.92 em \hbox{$\otimes$}}}}
\newcommand\oplusal{\mathop{\hbox{\raise 1.6 ex
  \hbox{$\scriptscriptstyle\mathrm{al}$}
\kern -0.92 em \hbox{$\oplus$}}}}
\newcommand\Gammal{\hbox{\raise 1.7 ex
\hbox{$\scriptscriptstyle\mathrm{al}$}\kern -0.50 em $\Gamma$}}

\let\al=\alpha \let\be=\beta  \let\ep=\epsilon

  \let\ga=\gamma 
 \let\la=\lambda

 \let\Ga=\Gamma \let\La=\Lambda

\newcommand{\caA}{{\mathcal A}}
\newcommand{\caB}{{\mathcal B}}
\newcommand{\caC}{{\mathcal C}}
\newcommand{\caD}{{\mathcal D}}
\newcommand{\caE}{{\mathcal E}}
\newcommand{\caF}{{\mathcal F}}
\newcommand{\caG}{{\mathcal G}}
\newcommand{\caH}{{\mathcal H}}

\newcommand{\caK}{{\mathcal K}}
\newcommand{\caL}{{\mathcal L}}

\newcommand{\caO}{{\mathcal O}}

\newcommand{\caR}{{\mathcal R}}
\newcommand{\caS}{{\mathcal S}}

\newcommand{\caV}{{\mathcal V}}
\newcommand{\caW}{{\mathcal W}}

\newcommand{\scrG}{{\mathscr G}}

\newcommand{\bbC}{{\mathbb C}}

\newcommand{\bbE}{{\mathbb E}}

\newcommand{\bbP}{{\mathbb P}}

\newcommand{\bbR}{{\mathbb R}}

\newcommand{\bbT}{{\mathbb T}}

\newcommand{\bbZ}{{\mathbb Z}}

\newcommand{\opunit}{\text{1}\kern-0.22em\text{l}}

\newcommand{\frz}{{\mathfrak z}}

\newcommand{\bss}{{\boldsymbol s}}

\newcommand{\bsu}{{\boldsymbol u}}

\newcommand{\non}{\nonumber}

\newcommand{\beq}{ \begin{equation} }
\newcommand{\eeq}{ \end{equation} }
\newcommand{\bet}{ \begin{theorem} }
\newcommand{\eet}{ \end{theorem} }

\newcommand{\baq}{\begin{eqnarray}}
\newcommand{\eaq}{\end{eqnarray}}

\newcommand{\norm}{ \|}

 \newcounter{smallarabics}
\newenvironment{arabicenumerate}
{\begin{list}{{\normalfont\textrm{\arabic{smallarabics})}}}
  {\usecounter{smallarabics}\setlength{\itemindent}{0cm}
  \setlength{\leftmargin}{5ex}\setlength{\labelwidth}{4ex}
  \setlength{\topsep}{0.75\parsep}\setlength{\partopsep}{0ex}
   \setlength{\itemsep}{0ex}}}
{\end{list}}

\newcounter{smallroman}

\newcommand{\ben}{\begin{arabicenumerate}}
\newcommand{\een}{\end{arabicenumerate}}

\newcommand{\dist}{\mathrm{dist}}

\newcommand{\eqlaw}{\stackrel{d}{=}}

\newcommand{{\banone}}{\scrG} 
\newcommand{{\gamzero}}{\ga_0} 
\newcommand{{\tengam}}{10\gamzero} 
\newcommand{{\fifteengam}}{15\gamzero} 
\newcommand{{\twentygam}}{20\gamzero} 
\newcommand{{\normba}}{\norm^{}_{\banone}}

\newcommand{\hf}{{_1\over^2}}

\title{Renormalization Group and Stochastic PDE's}

\author[A. Kupiainen]{Antti Kupiainen}%$^{1,2,3}$}
\address{University of Helsinki, Department of Mathematics and Statistics,
         P.O. Box 68 , FIN-00014 University of Helsinki, Finland}
\email{antti.kupiainen@helsinki.fi}
\thanks{Supported by Academy of Finland}
\date{\today}

\begin{document}
\maketitle

\begin{abstract}
We develop a Renormalization Group (RG) approach to the study of existence and uniqueness of solutions to stochastic partial differential equations driven by space-time white noise. As an example we prove well-posedness and independence of regularization for the $\phi^4$ model in three dimensions recently studied by Hairer. Our method is "Wilsonian": the RG allows to construct effective equations on successive space time scales. Renormalization is needed to control the parameters in these equations. In particular no theory of multiplication of distributions enters our approach.
 \end{abstract}

%\begin{center}
%\large{ \bf{Dynamics of $\phi^4_3$} } \\
%\vspace{15pt} \normalsize

%{\bf   A. Kupiainen\footnote{
%email: {\tt    antti.kupiainen@helsinki.fi  }}  }\\
%\vspace{10pt} 
%{\it   Department of Mathematics \\
% University of Helsinki \\ 
%P.O. Box 68, FIN-00014,  Finland 
%} \\

%\end{center}

%\vspace{20pt} \footnotesize \noindent {\bf Abstract: }  
\section{Introduction}
Nonlinear parabolic PDE's driven by a space time decorrelated noise are ubiquitous in physics.
Examples are thermal noise in fluid flow, random deposition in surface growth and stochastic dynamics for spin systems and field theories. These equations are of the form
\beq
 \partial_t u=\Delta u+F(u)+\Xi %, \ \ \ \varphi(0)=\varphi_0
  \label{eq: upde}
 %\non
\eeq
where $u(t,x)$ is defined on $\La\subset\bbR^d$, $F(u)$ is a function of $u$ and possibly  its derivatives which can also be non-local  and $\Xi$ is white noise on $\bbR\times \La$,
formally
\beq
 \bbE\ \Xi(t',x')\Xi(t,x)=\delta(t'-t)\delta(x'-x).
  \label{eq: white}
 %\non
\eeq
Usually in these problems one is interested in the behavior of solutions in large time and/or long distances in space. In particular one is interested in stationary states
and their scaling properties. These can be studied with regularized versions of the equations where the
noise is replaced by a mollified version that is smooth in small scales. Often one expects the large scale
behavior is insensitive to such regularization. 

From the mathematical point of view and sometimes also from the physical
one it is of interest to inquire the short time short distance properties i.e. the well-posedness of the equations without regularizations. Then one is encountering the problem that the solutions are expected to  have
very weak regularity, they are distributions, and it is not clear how to set up the solution theory for
the nonlinear equations in distribution spaces. 

Recently this problem was addressed by Martin Hairer \cite{hairer} who set up a solution theory for a class
of such equations, including the KPZ equation in one spatial dimension and the nonlinear heat equation with cubic nonlinearity in three spatial dimensions. The latter case was also addressed by Catellier and Chouk \cite{CC} based on the theory of paracontrolled distributions developed in \cite{GIP}.  The class of equations discussed in these works are
subcritical in the sense that the nonlinearity vanishes in small scales in the scaling that preserves the
linear and noise terms in the equation. In physics terminology these equations are {\it superrenormalizable}. This means the following.  Let $\Xi_\ep$ be a mollified noise with
short scale cutoff $\ep$. One can write a formal series solution to the mollified version of eq. \eqref{eq: upde} by starting with the solution $\eta_\ep(t)=\eta_\ep(t,\cdot)$ of the linear ($F=0$) equation and iterating: 
\beq
 u_\ep(t)=\eta_\ep(t)+\int_0^te^{(t-s)\Delta}F(\eta_\ep(s))ds+\dots
% \non
\eeq
Typically, the random fields (apart from
$\eta$) occurring in this expansion have
no limits as $\ep\to 0$: even when tested by smooth functions  their variances blow up. These divergencies are familiar from quantum field theory (QFT). Indeed, the correlation functions of $u_\ep$ have expressions in terms of Feynman diagrams
and as in QFT the divergencies can be cancelled in this formal expansion by a adding to $F$
extra  $\epsilon$-dependent terms, so-called {\it counter terms}. In QFT there is a well
defined algorithm for doing this and in the superrenormalizable case rendering the first few terms in the
expansion finite cures the divergences in the whole expansion. Hairer's work can be seen as reformulating this
perturbative renormalization theory as a rigorous solution theory for the subcritical equations.
It should be stressed that \cite{hairer} goes further by treating also rough non-random forces.

In QFT there is another approach to renormalization pioneered by K. Wilson in the 60's \cite{wilson}.
In Wilson's approach adapted to the SPDE one would not try to solve  equation \eqref{eq: upde}, call
it $\caE$,  directly
but rather go scale by scale starting from the scale $\ep$ and deriving {\it effective} equations $\caE_n$ 
for
larger scales $2^n\ep:=\ep_n$, $n=1,2,\dots$. Going from scale $\ep_n$ to $\ep_{n+1}$ is a problem
with $\caO(1)$ cutoff when transformed to dimensionless variables. This problem can be studied by a standard Banach fixed point method. The possible singularities of the original problem are present in the large $n$ behavior of   the corresponding effective equation. One views $n\to\caE_n$ as a dynamical system and
attempts to find an initial condition at $n=0$ i.e. modify $\caE$ so that if we fix the scale $\ep_n=\ep'$ and then let  $\ep\to 0$ (and as a consequence $n\to\infty$) the
effective equation at scale $\ep'$ has a limit. It turns out that controlling this limit for the effective equations allows one then to control the solution to the original equation \eqref{eq: upde}.

In this paper we carry out Wilson's renormalization group analysis to the cubic nonlinear heat equation
in three dimensions. This equation is a good test case since its renormalization is non trivial in the sense that a simple Wick ordering of the nonlinearity is not sufficient. Our analysis is robust in the sense that it
works for other subcritical cases like the KPZ equation. We prove almost sure local well-posedness for the mild (integral equation)  version of  \eqref{eq: upde} thereby recovering the results in \cite{hairer}. Our renormalization group method
is a combination of the one developed in \cite{bk} for parabolic PDE's and the one in \cite{bgk}
used for KAM theory.  A similar scale decomposition appears also in \cite{Unter}.

The content of the paper is as follows. In section 2 we define the model and state the result. The
RG formalism is set up in a heuristic fashion in Sections 3 and 4. Sections 5 and 6 discuss the leading perturbative solution and set up the fixed point problem for the remainder. Section 7 states the estimates
for the perturbative noise contributions and in Section 8 the functional spaces for RG are
defined and the fixed point problem solved.  The main result is proved in Section 9. Finally in Sections 10 and 11 estimates for the covariances of the various noise contributions are proved.

\section{The $\varphi_3^4$ model}
Let $\Xi(t,x)$ be space time white noise on $x\in\bbT^3$ i.e.  $\Xi=\dot\beta$ with $\beta(t,x)$ Brownian in time and white noise in space.
Given a realization of the noise $\Xi$ we want to make sense and solve the equation
\beq
 \partial_t\varphi=\Delta\varphi-\varphi^3-r\varphi+\Xi, \ \ \ \varphi(0)=\varphi_0 \label{eq: phi4pde}
% \non
\eeq
on some time interval $[0,\tau]$ and show $\tau>0$ almost surely. 

Due to the nonlinearity the equation \eqref{eq: phi4pde} is not well defined. We need to
define it through regularization. To do this we first formally write  it
in its integral equation form
\beq
\varphi=G(-\varphi^3-r\varphi+\Xi ) +e^{t\Delta}\varphi_0  \label{eq:inteq }
\eeq
where
\beq
(Gf)(t)=\int_0^te^{(t-s)\Delta}f(s)ds.   %\label{eq: }
\non
\eeq
(for  $f=\Xi$ this stands for  $
(G\xi)(t)=\int_0^te^{(t-s)\Delta}d\beta(s) $ ).  Next, introduce a regularization parameter $\ep>0$ and define 
\beq
  (G_\epsilon f)(t)=\int_0^t(1-\chi((t-s)/\ep^2))e^{(t-s)\Delta}f(s)ds .%=\chi_\epsilon\star G %
  \label{eq: Geps}
\eeq
where %$\chi_\epsilon(x)=\epsilon^{-3} \chi(x/\epsilon)$ and
$\chi\geq 0$ is a smooth bump, $\chi(t)=1$ for $t\in[0,1]$ and $\chi(t)=0$ for $t\in[2,\infty)$.
The  regularization of  \eqref{eq:inteq } with $\varphi_0=0$ is then defined to be 
\beq
\varphi=G_\epsilon(-\varphi^3-r_\epsilon\varphi+\Xi ) .%(1-\chi(t/\ep^2))
%+e^{t\Delta}\varphi_\ep 
 \label{eq: regpde }
%\non
\eeq
%where $\varphi_\ep$ is some regularization of the initial data. 
We look for $r_\epsilon$ such that  \eqref{eq: regpde } has a unique solution  $\varphi^{(\ep)}$
which converges as $\epsilon\to 0$ to a non trivial limit. Note that since only $t-s\geq \ep^2$
contribute in \eqref{eq: Geps} $G_\epsilon\Xi$ is a.s. smooth. %Also, for $t<\epsilon^2$ %there is no nonlinearity or noise:
%$\varphi(t)
%\equiv 0$. %=e^{t\Delta}\varphi_0$.
 
 Our main result is

\begin{theorem}\label{main result}  There exits $r_\ep$ s.t. the following holds. For almost all realizations of the white noise $\Xi$ there exists  $t(\Xi)>0$  such that the equation \eqref{eq: regpde } % with $\varphi_0=0$ or $\varphi_0=\eta_0$ 
has for all $\ep>0$ a unique smooth solution $\varphi^{(\ep)}(t,x)$,
$t\in [0,t(\Xi)]$ and there exists $\varphi\in \caD'( [0,t(\Xi)]\times \bbT^3)$ such that
$\varphi^{(\ep)}\to\varphi$ in  $\caD'( [0,t(\Xi)]\times \bbT^3)$. The limit $\varphi$ is independent
of the regularization $\chi$. 
\end{theorem}

\begin{remark}\label{rem: renorm} 
We will find that the renormalization parameter is given by
\beq
r_\epsilon=r+m_1\epsilon^{-1}+m_2\log\epsilon +m_3  \label{eq: reps}
\eeq
where the constants $m_1$ and $m_3$ depend on $\chi$ whereas the  $m_2$ is universal
i.e. independent on  $\chi$.
They of course agree with the mass renormalization needed to make sense  of the
formal stationary measure of  \eqref{eq: phi4pde}
\beq
\mu(d\phi)=e^{-\frac{1}{4}\int_{\bbT^3}\phi(x)^4 }%+\hf r\phi(x)^2)}
\nu(d\phi)  \label{eq: statmeas}
\eeq  
where $\nu$ is Gaussian measure with covariance $(-\Delta+r)^{-1}$ \cite{G, EO, GJ,  F, FO}. 
\end{remark}

\begin{remark}\label{rem: initial cond} 

This
 result  can be extended   a large class of initial conditions, deterministic or random.
%For deterministic ones one regularizes $e^{t\Delta}\varphi_0$ in   \eqref{eq:inteq }
%and treats it as a mean of
 As an example we consider the random case where $\phi_0=\eta_0$ where $\eta_0$ is the gaussian random
field on $\bbT^3$ with covariance $-\hf\Delta^{-1}$, independent of $\Xi$ (this is the stationary state of the linear equation). Then
Theorem \ref{main result} holds a.s. in the initial condition and $\Xi$,
see Remark \ref{rem: other initial}.
\end{remark}

\begin{remark}\label{rem: cutoff} One can as well introduce the short scale cutoff only to the spatial dependence of the noise by replacing $\Xi$ with  $\Xi_\ep:=\rho_\ep\star \Xi$ where $\rho_\ep(x)=\ep^{-3}\rho(x/\ep)$ with $\rho$ smooth nonnegative compactly supported bump integrating to one. Then the regularized equation is
\beq
\varphi=G(-\varphi^3-r_\epsilon\varphi+\Xi_\ep ) .%(1-\chi(t/\ep^2))
%+e^{t\Delta}\varphi_\ep 
 \label{eq: regpdespace }
%\non
\eeq
See Remark \ref{rem: cutoff1} for more discussion. 
This cutoff  has the advantage that \eqref{eq: regpdespace } represents a regular (Stochastic) PDE. 
\end{remark}

\section{Effective equation}

%For simplicity we start by the case $\varphi_0=0$ i.e. c
Consider the cutoff problem
\beq
\varphi=G_\epsilon (V(\varphi)+\Xi) \label{eq:epseq }
\eeq
for $\varphi(t,x)$ on $(t,x)\in [0,\tau]\times \bbT^3$ with 
$$V(\varphi)(t,x)=-\varphi^3(t,x)-r_\epsilon\varphi(t,x).$$

Let us attempt increasing the  cutoff $\epsilon$ to $\epsilon'>\epsilon$ 
by   solving the equation \eqref{eq:epseq } for scales between $\epsilon$ and $\epsilon'$.  
To do this split %Let $\la<1$ and split
\beq
G_\epsilon= G_{\epsilon'}+\Gamma_{\epsilon,\ep'}
%G_{\epsilon/\lambda}+\Gamma_{\epsilon,\la}
 %\label{eq: }
 \non
\eeq
with
\beq
(\Gamma_{\epsilon,\ep'}  f)(t)=\int_0^t(\chi((t-s)/{\ep'}^2)-\chi((t-s)/\ep^2))e^{(t-s)\Delta}f(s)ds. %\label{eq: }
 \non
\eeq
Thus $\Gamma_{\epsilon,\ep'}$ involves temporal scales between  $\epsilon^2$ and ${\epsilon'}^2$ (and due to the heat kernel spatial scales  between $\epsilon$ and $\epsilon'$).  Next, write
\beq
\varphi=\varphi'+Z \label{eq:phi'z }
 %\non
\eeq
and determine $Z=Z(\varphi')$ as a function of $\varphi'$ by solving the small scale equation
\beq
Z= \Gamma_{\epsilon,\ep'} (V(\varphi'+Z)+\Xi). \label{eq: Zeq}
\eeq
Eq.  \eqref{eq:epseq } will then hold provided
 $\varphi'$ is a solution to the "renormalized" equation
\beq
\varphi'= G_{\epsilon'}(V'(\varphi')+\Xi)\label{eq:reneq }
\eeq
where
\beq
V'(\varphi')= V(\varphi'+Z(\varphi')).%\label{eq: }
 \non
\eeq
Eq. \eqref{eq:reneq } is of the same form as  \eqref{eq:epseq } except that the cutoff has increased and $V$ is replaced by $V'$. Combining \eqref{eq: Zeq} and  \eqref{eq:reneq }
we see that the new $V'$ can be obtained by solving a fixed point equation
\beq
%V'(\phi')= V(\varphi'+\Gamma_{\epsilon,\ep'} (V'(\varphi')+\Xi)).%\label{eq: }
V'= V(\cdot+\Gamma_{\epsilon,\ep'} (V'+\Xi)).\label{eq:fixsdpoint }
% \non
\eeq
Finally, the solution of \eqref{eq:epseq } is gotten from \eqref{eq:phi'z } as
\beq
\varphi=\varphi'+\Gamma_{\epsilon,\ep'} (V'(\varphi')+\Xi) \label{eq:F'def }
 %\non
\eeq

Our aim is  to study the flow of the {\it effective equation $V'$ at scale} $\ep'$ as $\ep'$ increases from $\ep$ to $\tau^\hf$ where $[0,\tau]$ is the
time interval where we try to solve the original equation. It will be convenient to do this step by step. 
We fix a number $\la<1$ (taken to be small in the proof to kill numerical constants) and take the cutoff scale
\beq
\epsilon=\la^{N}.\label{eq:epsilonn }
\eeq 
The corresponding $V$ is denoted as
\beq
V^{(N)}(\varphi)=-\varphi^3-r_{\la^N}\varphi
.\label{eq:VN }
% \non
\eeq

Let $V^{(N)}_n$ be the solution of \eqref{eq:fixsdpoint } with $\ep'=\la^n$. We will construct these functions iteratively in $n$ i.e. derive the effective equation on scale $\la^{n-1}$
from that of $\la^n$:
\beq
V^{(N)}_{n-1}= V^{(N)}_n(\cdot+\Gamma_{\la^n,\la^{n-1}} (V^{(N)}_{n-1}+\Xi)).\label{eq:fixsdpointn }
% \non
\eeq
%$Z$ is the contribution to the solution of \eqref{eq:epseq } from spatial scales between $\ep$ and $\ep/\la$ and temporal scales  between $\ep^2$ and $\ep^2/\la^2$. The large scale field $\varphi'$ satisfies then the renormalized equation \eqref{eq:reneq } and we want to study the flow of $V'$ as $\la$ decreases.
The solution of \eqref{eq:epseq } is then also constructed iteratively: let
 \beq
F^{(N)}_N(\varphi):=\varphi%\label{eq: }
\non
\eeq 
and define
 \beq
F^{(N)}_{n-1}=F^{(N)}_n(\cdot+\Gamma_{\la^n,\la^{n-1}} (V^{(N)}_{n-1}+\Xi)) .\label{eq: Fnite}
\eeq 
Then the solution of \eqref{eq:epseq } is
 \beq
\varphi=F^{(N)}_n(\varphi_n) .%\label{eq: }
\non
\eeq 
where $\varphi_n$ solves
\beq
\varphi_n= G_{\la^n}(V^{(N)}_n(\varphi_n)+\Xi).\label{eq:phineq }
\eeq
Finally we want to control the limit $N\to\infty$ where the regularization is removed i.e.
construct the limits $V_n=\lim_{N\to\infty}V^{(N)}_n$ and $F_n=\lim_{N\to\infty}F^{(N)}_n$
which are then shown to describe the solution of \eqref{eq: phi4pde} on time interval
$[0,\la^{2n}]$. How long this time interval will be i.e. how small $n$ we can reach in the
iteration depends on the realization of the noise. In informal terms, let $A_m$ be the event that
these limits exist for $n\geq m$. We will show
 \beq
\bbP(A_m)\geq 1-\caO(\la^{Rm}).\label{eq:probaest1 }
\eeq
with large $R$ i.e. the set of noise s.t. the equation \eqref{eq: phi4pde} is well posed on time interval
$[0,\tau]$ has probability $1-\caO(\tau^{R})$. %For a precise statement see Theorem \ref{}.

\section{Renormalization group}

The equation \eqref{eq:fixsdpointn } deals with scales between $\la^n$ and $\la^{n-1}$.
 Instead of letting the scale of the equations vary it will be be more convenient to rescale everything to fixed scale (of order unity) after which we need to iterate a $\caO(1)$-scale
problem. This is the "Wilsonian" approach to Renormalization Group.
 
Let us define the space time scaling $s_\mu$ by
\beq
 (s_\mu f)(t,x)=\mu^\hf f(\mu^2t,\mu x).%\label{eq: }
 \non
\eeq
The Green function of the heat equation and the space time white noise transform
in a simple way under this scaling: 
\beq
s_\mu Gs_\mu^{-1} =\mu^2G,\ \ \ s_\mu G_\ep s_\mu^{-1} =\mu^2G_{\ep/\mu}, \ \ \ s_\mu \Xi \eqlaw \mu^{-2}\Xi_\mu %\label{eq: }
 \non
\eeq
as one can easily verify by a simple changes of variables. Here $\Xi_\mu$ is space time white noise on $\bbR\times \mu^{-1}\bbT^3$. Set 
\beq
\phi_{n}=s_{\la^n}\varphi_n \label{eq: phindef}
%\label{eq: }
\eeq
where $\varphi_n$ solves \eqref{eq:phineq }.
Note that $\phi_{n}(t,x)$ is defined on $x\in   \bbT_n$ with
\beq
\bbT_n:= \la^{-n}\bbT^3%\Lambda_n:=[0,\tau_n]\times\bbT_n^d,\ \ \tau_n:=\tau/\epsilon_n^2,\ \ \ \bbT_n=\epsilon_n^{-1}\bbT
\label{eq: Lambdan}
\eeq
Since $\varphi_n$ was interpreted as a field involving spatial scales on $[\la^n,1]$ 
 $\phi_n$ may be thought   as a field involving scales on $[1,\la^{-n}]$.  
 
In these new variables eq. \eqref{eq:phineq } becomes
\beq
 \label{eq:neweq1 }
\phi_n= G_1 (v^{(N)}_n(\phi_n)+\xi_n)%\non
\eeq
provided we define 
\beq
v^{(N)}_n:=\la^{2n}s_{\la^{n}}\circ V^{(N)}_n\circ  s_{\la^n}^{-1}%\label{eq: }
 \non
\eeq
and
\beq
\xi_n:=\la^{2n}s_{\la^n}\Xi.
%\label{eq: }
 \non
\eeq
$\xi_n$ is distributed as a space time white noise on $\bbR\times\bbT_n$. Defining
\beq
f^{(N)}_n:=s_{\la^{n}}\circ 
F^{(N)}_n\circ s_{\la^{n}}^{-1}%\label{eq: }
 \non
\eeq
the solution of \eqref{eq:epseq }  is given by
 \beq
%s_{\ep_n}
\varphi=s_{\la^{n}}^{-1}f^{(N)}_n(\phi_n) .\label{eq: finalsolution}
%\non
\eeq 
%Letting $\tau=\la^{2N}$ we then have $\tau_N=1$ and
%\beq
%\varphi=f_N(\phi_N) .\label{eq: finalsolution}
%\eeq 

The iterative equations  \eqref{eq:fixsdpointn } and \eqref{eq: Fnite} have their
unit scale counterparts: 
recalling \eqref{eq:epsilonn }, denoting $s:=s_\la$ and using $s_{\la^{n+1}}=s\circ s_{\la^n}$ equation \eqref{eq:fixsdpointn } becomes
\beq
v^{(N)}_{n-1}(\phi)=\la^{-2}s^{-1}v^{(N)}_n(s(\phi+\Gamma_{}(v^{(N)}_{n-1}(\phi)+\xi_{n-1})))% :=\caF_n(v^{(N)}_n,v^{(N)}_{n+1})
 \label{eq: vn+1new}
%\label{eq: }
\eeq
where 
\beq
(\Gamma  f)(t)=\int_0^t(\chi(t-s)-\chi((t-s)/\la^2))e^{(t-s)\Delta}f(s)ds.
\label{eq:Gala }
\eeq
Note the integrand is supported on  the interval $[\la^2,2]$.
%\beq
%\chi_\la(s)=\chi(s)-\chi(s/\la^2)\label{eq:chila }
%\eeq is a smooth function supported in the interval $[\la^2,2]$.

Eq.  \eqref{eq: Fnite} in turn becomes
\beq
f^{(N)}_{n-1}(\phi)=s^{-1}
f^{(N)}_n(s(\phi+\Gamma_{}(v^{(N)}_{n-1}(\phi)+\xi_{n-1})))  \label{eq: fn+1new}.
%\label{eq: }
\eeq
%Note that eq. \eqref{eq: vn+1new} is an equation for $v
Our task then is to solve the equation  \eqref{eq: vn+1new}  for $v^{(N)}_{n-1}$ to obtain the RG map
\beq
v^{(N)}_{n-1}=\caR_nv^{(N)}_n \label{eq: Rn}
%\label{eq: }
\eeq
and then 
iterate this    and \eqref{eq: fn+1new}
starting with
\baq
 v^{(N)}_N(\phi)&=&-\la^N\phi^3-\la^{2N}r_{\la^N}  \phi=-\la^N\phi^3-(\la^{N}m_1+\la^{2N}(m_2\log\la^N+m_3))\phi
 \label{eq: first v}\\
  f^{(N)}_N(\phi)&=&\phi\label{eq: first f}
\eaq
If the noise is in the set $A_m$ (to be defined) we then show the functions $v^{(N)}_{n}$ and  $f^{(N)}_{n}$ 
have limits as $N\to\infty$ and $n\geq m$ and allow to construct the solution to our original equation
on time interval $[0,\la^{2m}]$.

Equations  \eqref{eq: vn+1new}, \eqref{eq: fn+1new} and  \eqref{eq:neweq1 } involve the operators $\Gamma$ and  $G_1$ respectively. These operators are  infinitely smoothing and their kernels have fast decay in space time.
  In particular  the noise $\zeta=\Gamma\xi_n$ entering equations  \eqref{eq: vn+1new} and \eqref{eq: fn+1new} has a smooth covariance which is
short range in time
\beq
\bbE \zeta(t,x)\zeta(s,y)= 0\ \ \ {if} \ \ |t-s|>2\la^{-2},\label{eq: F1}
\eeq
and it has gaussian decay in space. Hence the fixed point problem  \eqref{eq: vn+1new} turns out to be quite easy.

As usual in RG studies one needs to keep track of the leading "relevant" terms of
$v^{(N)}_n$ which are revealed by a first and second order perturbative study of   \eqref{eq: vn+1new} to which we turn now.
%At each step of iteration the spatial volume is reduced by $\la$ and after $N=\log\epsilon/\log\la$ steps we reach the a torus of unit volume.

\section{Linearized renormalization group}

We will now study the fixed point equation \eqref{eq: vn+1new} to first order in $v$. %Let us study the linearization of the RG map \eqref{eq: Rn} i.e.
Define the map
\beq
(\caL_nv)(\phi):=\la^{-2}s^{-1}v(s(\phi+\Gamma\xi_{n-1}))
\label{eq:linearrg}
\eeq
Then  \eqref{eq: vn+1new} can be written as
\beq
v^{(N)}_{n-1}(\phi)=(\caL_nv^{(N)}_n)(\phi+\Gamma_{}v^{(N)}_{n-1}(\phi))% :=\caF_n(v^{(N)}_n,v^{(N)}_{n+1})
 \label{eq: vn+1new1}
%\label{eq: }
\eeq
and we see that $\caL_n=D\caR_n(0)$, the derivative of the RG map. The linear flow $u_{n-1}=\caL_nu_n$  from scale $N$ to scale $n$
is easy to solve by just replacing $\la$ by $\la^{N-n}$:
\beq
u_n=(\caL_{n+1}\caL_{n+2}\dots \caL_{N} u_N)(\phi)=\la^{-2(N-n)}s^{n-N}u_N(s^{N-n}(\phi+\eta^{(N)}_n))
\non
\eeq
where 
\beq
\eta^{(N)}_{n}=\Gamma^N_n\xi_{n}.\label{eq:etandefinition }
%\non
\eeq
Here  $\xi_{n}$ is space time white noise on $\bbT_n=  \la^{-n}\bbT^3$ 
and $\Gamma^N_n$ is given by \eqref{eq:Gala } with $\la $ replaced by $\la^{N-n} $ i.e. its integral kernel is
\beq
 \Gamma^N_n(t,s,x,y)=\chi_{N-n}(t-s)%\int_0^t
 H_n(t-s,x-y) %e^{(t-s)\Delta}\xi_n(s)ds
 \label{eq: gamman-ndef}
%\non
\eeq 
where we denoted 
$$
H_n(t,x-y)=e^{t\Delta}(x,y)
$$  
the heat kernel on $\bbT_n$ and
\beq
\chi_{N-n}(s):=\chi(s)-\chi(\la^{-2(N-n)}s)\label{eq: gammaN-n}
%\non
\eeq 
is a smooth indicator of the interval $ [\la^{2(N-n)},2]$.

The linearized flow is especially simple for a local $u$ as the one we start with \eqref{eq: first v}.
For $u_N=\phi^k$ we get  
$$u_n=\la^{n(k-5)/2}(\phi+\eta^{(N)}_n)^k$$
and so we have "eigenfunctions"
\beq
\caL_n(\phi+\eta^{(N)}_n)^k=\la^{(k-5)/2}(\phi+\eta^{(N)}_{n-1})^k .%\label{eq: }
\non
\eeq
For $k<5$ these are "relevant", for $k>5$ they are "irrelevant" and for $k=5$ "marginal".

%Let $G_n(t,x-y)=e^{t\Delta}(x,y)$ be the heat kernel on $\bbT_n$. 
The covariance of  $\eta^{(N)}_n$ is readily obtained from \eqref{eq: gamman-ndef} (let $t'\geq t$):
\baq
 \bbE \eta^{(N)}_{n}(t',x') \eta^{(N)}_{n}(t,x)&=&%G_n(t'-t)\int_0^tG_n(2s)(\chi(s)-\chi(\la^{-2n}s))^2dsG_n(t'-t)
%e^{(t'-t)\Delta}
\int_0^t%e^{2s\Delta}
H_n(t'-t+2s,x'-x)\chi_{N-n}(t'-t+s)\chi_{N-n}(s)ds\non\\
&:=&C^{(N)}_n(t',t, x',x)
\label{eq:Cndef }
\eaq
 In particular we have
\beq
\bbE \eta^{(N)}_{n}(t,x)^2=\int_0^tH_n(2s,0)\chi_{N-n}(s)^2ds.%=\int_0^t(4\pi s)^{-3/2}(\chi(s)-\chi(\la^{-2n}s))^2ds%=\la^{-n}\rho(\frac{2}{3}(4\pi)^{3/2}\la^{-n} +\caO(1)%
\label{eq:etasq }
\eeq
This integral diverges as $N-n\to\infty$ and is the source of the first renormalization constant
 in \eqref{eq: first v}. 
 
 We need to study the
$N$ and $\chi$ dependence of the solution to \eqref{eq: regpde } . Since these dependencies are very similar we deal with them together. Thus let
$\Gamma'^{(N)}_{n}$ the operator \eqref{eq: gamman-ndef}  where the lower cutoff 
in \eqref{eq: gammaN-n} is modified to another bump $\chi'$:
\beq
\chi'_{N-n}(s)=\chi(s)-\chi'(\la^{-2(N-n)}s).%\int_0^t
\label{eq: gammaN-ndef}
%\non
\eeq 
 Varying $\chi'$ allows to study cutoff dependence of our scheme. Taking  $\chi'(s)=\chi(\la^{-2} s)$ in turn implies $\Gamma'^{(N)}_{n}=\Gamma^{(N+1)}_{n}$ and this allows us to study the  $N$ dependence and the convergence as $N\to \infty$. The following lemma controls these dependences for 
\eqref{eq:etasq }:

\begin{lemma}\label{lem: heatkernel} Let 
\beq
\rho:= \int_0^\infty(8\pi s)^{-3/2}(1-\chi(s)^2)ds.\label{eq: achidef}
\eeq
Then, %for $t>\la^{2(N-n)}$
\beq
\bbE \eta^{(N)}_{n}(t,x)^2=\la^{-(N-n)}\rho+\delta_n^{(N)}(t)%\caO(\la^{N-n}/\sqrt t))+\caO(\la^{3n})
%\label{eq: }
\non
\eeq
with
\beq
|\delta_n^{(N)}(t)|\leq C(1+t^{-\hf})
\label{eq: deltaNnbound}
\eeq
and 
\beq
|\delta_n^{(N)}(t)-\delta_n'^{(N)}(t)|\leq C(t^{-\hf}1_{[0,2\la^{2(N-n)}]}(t)+e^{-c\la^{-2N}})\|\chi-\chi'\|_\infty
\label{eq: deltaNnN'bound}
\eeq
\end{lemma}

\noindent The Lemma is proved in Section \ref{se:kernel estimates}.
We will fix in \eqref{eq: first v}  the first renormalization constant
\beq
a=-3\rho
\label{eq: 1srgconstant}
%\non
\eeq
Defining
\beq
\rho_k=\la^{-k}\rho,
\label{eq: 1srgconstantk}
%\non
\eeq
the first order solution to our problem is
\beq
u^{(N)}_n:=-\la^n((\phi+\eta^{(N)}_n)^3+3\rho_{N-n}(\phi+\eta^{(N)}_n)).
\label{eq: undef}
\eeq

\begin{remark}\label{rem: relevant} 
Note that since we have the factor $\la^n$ in $u^{(N)}_n$ the counting of what terms 
are relevant, marginal or irrelevant  depends on the order in $\la^n$. Thus $\la^n(\phi+\eta^{(N)}_n)^k$
is  relevant for $k<3$, marginal  for $k=3$ and irrelevant   for $k>3$.
Similarly, $\la^{2n}(\phi+\eta^{(N)}_n)^k$ is marginal for $k=1$ which is the source of the
renormalisation constant  $b$ in \eqref{eq: reps}. The terms of order $\la^{3n}$
are all irrelevant. For a precise statement, see Proposition \ref{prop: linrgnorm}. 
\end{remark}

\begin{remark}\label{rem: other initial} Consider the random initial condition discussed in Remark \ref{rem: initial cond}. We realize it in terms of the white noise on $(-\infty,0]\times \bbT_n$. In a regularized form we  replace $e^{t\Delta}\varphi_0  $ in\eqref{eq:inteq } by
\beq
\int_{-\infty}^0(1-\chi((t-s)/\ep^2))e^{(t-s)\Delta}\Xi(s)ds. \label{eq: regpdeini }
%\non
\eeq
This initial condition  can be absorbed to $\eta^{(N)}_n$. Indeed, the
covariance \eqref{eq:Cndef } is just replaced by the stationary one
\beq
 \bbE \eta^{(N)}_{n}(t',x') \eta^{(N)}_{n}(t,x)= \int_{-\infty}^t%e^{2s\Delta}
H_n(t'-t+2s,x'-x)\chi_{N-n}(t'-t+s)\chi_{N-n}(s)ds
\label{eq:Cndefstat }
\eeq
and in particular  $\delta^{(N)}_{n}=0$ which makes the analysis in Section
\ref{se:noise} actually less messy.
\end{remark}

% \begin{remark}\label{rem: Wick} 
 %The expression \eqref{eq: undef} is close to a  "normal ordered" one. Let
%\beq
%\tilde u^{(N)}_n:=\la^n((\phi+\eta^{(N)}_n)^3-3 (\phi+\eta^{(N)}_n)\bbE(\eta^{(N)}_n)^2)
%\label{eq: tildeundef1}
%\eeq
%and 
 %\beq
%\delta_n(t):=\la^n(\bbE(\eta^{(N)}_n(t,x))^2-3\rho_{N-n}).\label{eq: deltan}\eeq
%Then 
%\beq
%u^{(N)}_n=\tilde u^{(N)}_n+3\delta_n(\phi+\eta^{(N)}_n).
%\label{eq: oepsilon}
%\eeq
%In virtue of Lemma \ref{lem: heatkernel} the second term will give a small contribution.
%\end{remark}

\section{Second order calculation}

We will solve equation  \eqref{eq: vn+1new} by a fixed point argument in a suitable space of
$v^{(N)}_n$. Before going to that we need to spell out explicitly the leading relevant (in the RG sense)
terms. By Remark \ref{rem: relevant}  this requires looking at the second order terms in $v^{(N)}_n$.
To avoid too heavy notation we will drop the superscript ${}^{(N)}$ in $v^{(N)}_n$, $\eta^{(N)}_n$ and other expressions unless needed for clarity. %we specifically want to discuss the $N$ dependence.

Let us  first separate in  \eqref{eq: vn+1new} the linear part:
\beq
v_{n-1}=\caL_nv_n +\caG_n(v_n,v_{n-1})
\label{eq:newfp }
%\non
\eeq
where we defined 
\beq
\caG_n(v,\bar v)(\phi)=(\caL_nv)(\phi+\Gamma_{}\bar v(\phi))-(\caL_nv)(\phi)
\label{eq: Gdef1}
%\label{eq: }
\eeq
Recalling the first order expression $u_n$ in \eqref{eq: undef}
we write
\beq
v_n=u_n+w_n\non
%\label{eq: }
\eeq
so that $w_n$ satisfies
\beq
w_{n-1}=\caL_nw_n +\caG_n(u_n+w_n,u_{n-1}+w_{n-1}), 
\label{eq:newfp }
%\non
\eeq
with the initial condition
\beq
w_N(\phi)=-\la^{2N}(m_2\log\la^N+m_3)\phi.
\label{eq:newfpic }
%\non
\eeq
The reader should think about $u_n$ as $\caO(\la^{n})$ and $w_n$ as $\caO(\la^{2n})$.

Next, we separate from  the $\caG_n%(u^{(N)}_n,u^{(N)}_{n-1})
$-term in \eqref{eq:newfp } the $\caO(\la^{2n})$ contribution. Since $\caL_nu_n=u_{n-1}$ we have
\beq
\caG_n(u_n,u_{n-1})(\phi)=u_{n-1}(\phi+\Gamma_{}u_{n-1}(\phi))-u_{n-1}(\phi)
\label{eq: Gdef1}
%\label{eq: }
\eeq
which to $\caO(\la^{n})$ equals $Du_{n-1} \Gamma u_{n-1}$ where
%\beq
%\tilde\caG_n(v,\tilde v):=\caG_n(v,\tilde v)-D(\caL_nv)\Ga\tilde v
%\label{eq: Gdef1}
%\label{eq: }
%\eeq
%\beq
%\caG_n(u^{(N)}_n,u^{(N)}_{n-1})=%\epsilon_n\epsilon_{n+1}
%D(\caL_nu^{(N)}_n)\Gamma u^{(N)}_{n-1}+\tilde\caG_n(u^{(N)}_n,u^{(N)}_{n-1})%\caO(\la^{3n})
%. \label{eq: 2ndorder}
%\non
%\eeq
%Since $\caL_nu_n=u_{n-1}$ we get
%$$Du^{}_n=3\la^n((\phi+\eta^{}_n)^2-\la^{-n}a_\chi%\bbE\eta_n^2).
%$$ 
%The first term in \eqref{eq: 2ndorder} equals
%$$
%\la^{-2}s^{-1}Du^{}_n(s(\cdot+\Gamma\xi_{n-1}))\la^{\hf}\Gamma u^{}_{n-1}
%$$
%and since
%$$
%\la^{-2}s^{-1}Du_n(s(\cdot+\Gamma\xi_{n-1}))=\caL_nDu_n=\la^{-\hf} Du_{n-1}
%$$
%\beq
%\tilde\caG_n%(u_n,u_{n-1}) 
% =\caG_n(u_n,u_{n-1})-Du_{n-1} \Gamma u_{n-1}.%+\caO(\la^{3n}).%\label{eq: }\non\eeq
\beq
Du_{n-1}=-3\la^{n-1}((\phi+\eta_{n-1})^2-\rho_{N-n+1}).%\bbE\eta_n^2)%\label{eq: }
\non
\eeq
%$\tilde\caG_n$  is  thus $\caO(\la^{3n})$.  Summarizing:
Hence
\beq
w_{n-1}=\caL_nw_n+Du_{n-1} \Gamma u_{n-1}+\caF_n(w_{n-1})% ,\ \ \ \nu^{}_{N}=v_N%
\label{eq:weua }
%\non
\eeq
where
\beq
\caF_n(w_{n-1})=\caG_n(u_n+w_n,u_{n-1}+w_{n-1})-Du_{n-1} \Gamma u_{n-1}%\caG(u_n,u_{n-1}))%+\tilde\caG_n.%(u_n,u_{n-1}).   
\label{eq:Fnequa }
%\non
\eeq
$\caF_n$ is $\caO(\la^{3n})$ and will turn out to  be irrelevant under the RG (i.e. it will contract in a suitable norm under the linear RG map $\caL_n$). 

It is useful to solve  \eqref{eq:weua } without the $\caF_n$ term: let $U_n$ satisfy
\beq
U_{n-1}=\caL_nU_n+Du_{n-1} \Gamma u_{n-1},\ \ \ U_{N}=-\la^{2N}(m_2\log\la^N+m_3)\phi
\label{eq:Uiteration and ic }
%\non
\eeq
The solution is 
\beq
 U_n=Du_{n} \Gamma^N_nu_{n}-\la^{2n}(m_2\log\la^N+m_3)(\phi+\eta_n)\label{eq: v1n}
\eeq
where $\Gamma^N_n$ is defined in  \eqref{eq: gammaN-n}.
%\beq
% \Gamma^N_n(t,s)=e^{(t-s)\Delta}\chi_{N-n}(t-s)%(\chi(t-s)-\chi(\la^{-2n}(t-s)))
 %\label{eq:Gammalandef }
%\non
%\eeq
This is seen by just doing one step RG with $\la$ replaced by $\la^{N-n}$. 
Now write
\beq
w_n= U_n+\nu_n\label{eq:vndeco }
%\non
\eeq
so that $\nu_{n-1}$ satisfies
the equation
\beq
\nu_{n-1}= \caL_n\nu_{n} +\caF_n(U_{n-1}+ \nu_{n-1}),\ \ \ \nu_N=0\label{eq: wfp}
\eeq

Eq. \eqref{eq: wfp} is a fixed point equation that we will solve by contraction in a suitable space.
%$\nu_{n}$ turns out to be irrelevant (contracting) under the linear RG $ \caL_n$ so we need to show the nonlinear terms will be small perturbations.   We shall rewrite  \eqref{eq: v1n} using $\log\la^N=\log\la^n+\log\la^
%{N-n}$  as
%\beq
% U_n=(Du_{n} \Gamma^N_n u_{n}-b\la^{2n}\log\la^
%{N-n}(\phi+\eta_n))+b\la^{2n}\log\la^{n}(\phi+\eta_n).
% \label{eq: Undefin}
%\eeq
%In the next Section we will show that $U^{(N)}_n$ is be a well behaved random field as $N\to\infty$ provided we
%choose the renormalisation constant $b$ suitably. 

\section{Noise estimates}

The time of existence for the solution depends on the size of the noise. In this Section we state
probabilistic estimates for this size. 

The noise enters the fixed point equation \eqref{eq: wfp}  in
the form $\Gamma\xi_{n-1}$ that enters the definition of $\caL_n$ in \eqref{eq:linearrg} and  in the polynomials $u_n$ \eqref{eq: undef} and $U_n$ \eqref{eq: v1n} of the random field
$\eta_n$ \eqref{eq:etandefinition }.  According to Remark \ref{rem: relevant} in the second order term $U_n$  only   the constant in $\phi$ term should be relevant under the linear RG and the linear in  $\phi$ term should be marginal (neutral), the rest being irrelevant (contracting).
We'll see this indeed is the case and accordingly write 
\beq
U_n(\phi)=U_n(0)+DU_n(0)\phi+V_n(\phi) 
\label{eq: U_n}
\eeq
where explicitly
\beq
U_n(0)=3\la^{2n}(\eta_n^2-\rho_{N-n})\Gamma^N_n(\eta_n^3-3\rho_{N-n}\eta_n)
-\la^{2n}(m_2\log\la^N+m_3)\eta_n:=\la^{2n}\omega_n
 \label{eq:U_n(0) }
%\non
\eeq
and
\beq
(DU_n(0)\phi)(t,x)=\la^{2n}\frz_n(t,x)\phi(t,x)+\la^{2n}\int z_n(t,x,s,y)\phi(s,y)dsdy
 \label{eq:DU_n(0) }
%\non
\eeq
where
\beq
\frz_n=6\eta_n\Gamma^N_n(\eta_n^3-3\rho_{N-n}\eta_n)
 \label{eq:zeta_n }
%\non
\eeq
and
\beq
z_n=
9
(\eta_n^2-\rho_{N-n})\Gamma^N_n(\eta_n^2-\rho_{N-n})-m_2\log\la^N-m_3
 \label{eq: z_n}
%\non
\eeq
(here $\eta_n^2-\rho_{N-n}$ is viewed as a  multiplication operator).

The random fields whose size we need to constrain probabilistically are then 
\beq
 \eta_n, \ \ \eta_n^2-\rho_{N-n},\ \ \eta_n^3-3\rho_{N-n}\eta_n,\ \ \omega_n\ \ \frz_n,\ \ z_n
 \label{eq: noisefields}
%\non
\eeq
They belong to the Wiener chaos of white noise of bounded order and their size and regularity are controlled by studying their covariances. 
 For finite cutoff parameter $N$ these noise fields are a.s. smooth but in the limit $N\to\infty$ they become distribution valued. 
These fields enter in the RG iteration  \eqref{eq: vn+1new1} in the
combination  $\Gamma_{}v_{n-1}$ i.e. they are always acted upon by the operator $\Gamma$ which is infinitely smoothing. Therefore we estimate their size in suitable (negative index) Sobolev type norms which we now define. In addition to the fields \eqref{eq: noisefields} we also need to constrain the
Gaussian field $\Ga\xi_n$.

Let $K_1$ be the operator $(-\partial_t^2+1)^{-1}$  on $L^2(\bbR)$ i.e. it has the  integral kernel
\beq
 K_1(t,s)=\hf e^{-|t-t'|}
 \label{eq: K_1def}
%\non
\eeq
%i.e. $K_1$ %=(-\partial_t^2+1)^{-1}$ 
% is the Green function of $-\partial_t^2+1$ with boundary conditions
%$f(0)=f'(0)$. 
Let $K_2=(-\Delta+1)^{-2}$ on $L^2(\bbT_n)$ which has a continuous kernel $ K_2(x,y)=
K_2(x-y)$ satisfying
\beq
 K_2(x)\leq Ce^{-|x|}.\label{eq: K2bounded}
%\non
\eeq
Set
\beq
 K:=K_1K_2. %(-\partial_t^2+1)^{-1}%(-\Delta+1)^{-1}\label{eq: Kdef}
 %(-\Delta+1)^{-2}:=K_1K_2.
 \label{eq: Kdef}
%\non
\eeq
%Note for later use that $K$ is a positive operator on $L^2(\bbR_+\times\bbT_n)$ and it has  a continuous positive   integral kernel
%$K(t-s,x-y)$ satisfying
%\beq
 %K(t,x)\leq Ce^{-|t|-|x|}.\label{eq: Kbounded}
%\non
%\eeq
% i.e. $K$ is the operator on $L^2(\bbR)$ with integral kernel $\hf e^{-|t-t'|}$.
%Let $I_n:=[0,\la^{-2n}]$ and 
%\beq
% \label{eq: lambdandef}
%\Lambda_n:=I_n\times \bbT_n%\non
%\eeq
 Define $\caV_n$ to be be the completion of $C_0^\infty(\bbR_+\times\bbT_n)$ with the norm
\beq
\|v\|_{\caV_n}=\sup_{i}
\|Kv\|_%\infty %
{L^2(c_i)}
\eeq
%where on the RHS $v\in L^2(\Lambda_n)$ is naturally viewed as an element of  $L^2(\bbR\times\bbT_n^3)$ (i.e. $1_{I_n}v$).
where $c_i$ is the unit cube centered at $i\in \bbZ\times(\bbZ^3\cap\bbT_n)$.
To deal with the bi-local field $z_n$ in \eqref{eq: z_n} we define for $z(t,x,s,y)$ in $C_0^\infty(\bbR_+\times\bbT_n\times \bbR_+\times\bbT_n)$ 
 %)\in L^2(\La_n\times\La_n)$
\beq
\|z\|_{\caV_n}=%\sup_{t,s,x,y}|(K\otimes Kz)(t,x,s,y)|e^{-|t-s|-|x-y|}%
\sup_{i}\sum_j\|K\otimes Kz\|_{L^2(c_i\times c_j)}
\eeq

Now we can specify the admissible set of noise. Let $\ga>0$  and  define events $\caA_m$, $m>0$  in the probability space of the space time white noise $\Xi$
as follows. Let $\zeta^{(N)}_n$ denote any one of the fields \eqref{eq: noisefields}. We want to constrain the size of  $\zeta^{(N)}_n$  on the time interval $[0,\tau_{n-m}]$ where we will denote
$\tau_{n-m}=\la^{-2(n-m)}$. To do this choose a smooth bump $h$ on $\bbR$ with $h(t)=1$ for $t\leq -\la^2$ and $h(t)=0$ for $t\geq-\hf\la^2$ and set 
$h_{k}(t)=h(t-\tau_k)$ so that  $h_{k}(t)=1$ for $t\leq\tau_{k}-\la^2$ and $h_{k}(t)=0$ for $t\geq\tau_{k}-\hf\la^2$ (the reason for these strange choices will become clear in Section \ref{se:fpp}). The first condition on $\caA_m$ is that for all $N\geq n\geq m$
the following hold: 
\beq
%\caA_m =\{\Xi:  \|\rho^{(N)}_n\|_{\caV_n}, \|\Gamma \xi_n\|_{\Phi_n}\leq \la^{-\ga n}\ \ \forall n\geq m, N\geq n\}
\|h_{n-m}\zeta^{(N)}_n\|_{\caV_n}\leq\la^{-\ga n}% \,\, \,\, \,\, \,\, \,\, \,\, \,\,(\caA_m(1))%\ \ \forall n\geq m, N\geq n
%\non\\
\label{eq: amevent1}
%\non
\eeq

We need also to control the
$N$ and $\chi$ dependence of the  noise fields $\zeta^{(N)}_n$ . Recall that we can study both by varying the lower cutoff in the operator $\Gamma^{(N)}_{n}$ in  \eqref{eq: gamman-ndef}. We denote by  $\zeta'^{(N)}_n$ any of the resulting noise fields.  Our second condition on  $\caA_m$ is that for all $N\geq n\geq m$  and all cutoff functions $\chi, \chi'$ with bounded $C^1$ norm
\beq
\|h_{n-m}(\zeta'^{(N)}_n-\zeta^{(N)}_n)\|_{\caV_n}\ \leq  \la^{\ga(N-n)}\la^{-\ga n}. %\|\chi-\chi'\|_\infty.
\label{eq: amevent2}
%\non
\eeq 
The final condition concerns the fields $\Ga\xi_n$ entering the RG iteration \eqref{eq: vn+1new}. Note that these fields are $N$ independent and smooth and we impose on them a smoothness condition given in \eqref{eq: Gammaxibound}.
We have:
\begin{proposition}\label{prop: mainproba} There exist renormalization constants $m_2$ and $m_3$  %universal (i.e. independent on the cutoff$\chi$) 
 such that for some $\ga>0$ almost surely $\caA_m$ holds for some $m<\infty$.
\end{proposition} 

On  $\caA_m$ we will  control the RG iteration for scales $n\geq m$.
 This will enable us to solve the equation \eqref{eq: phi4pde} on the time interval
$[0,\la^{2m}]$.

\section{Fixed point problem}\label{se:fpp}

 %$\Xi=\dot\beta$ with $\beta(t,x)$ Brownian in time and white noise in space. 
We will now fix the noise $\Xi\in \caA_m$ for some $m> 0$ and set up a suitable space of functions $\nu_n(\phi)$  for the  fixed point
problems \eqref{eq: wfp} for $n\geq m$. 

Since the noise contributions take values in $\caV_n$ we let $\nu_n$ take values there as well.
 The noise enters in \eqref{eq: vn+1new} in the argument of $v^{}_n$ in the
combination  $\Gamma_{}(v_{n-1}+\xi_{n-1})$. Since $\Gamma_{}$ is infinitely smoothing this means that we
may take the domain of  $\nu_n(\phi)$  to consist of suitably smooth functions $\phi$. Since the noise contributions become distributions in the limit $N\to\infty$ and enter multiplicatively with $\phi$ e.g. in \eqref{eq: undef} we need to
match the smoothness condition for $\phi$ with that of the noise. Finally,
since \eqref{eq:neweq1 } implies $\phi_n\equiv 0$ on $[0,1]$ we let the $\phi$
be defined on $[1,\tau_{n-m}]$.%Since the problem is homogenous in space and time the norms will be of $L^\infty$-type in long space time distances.

With these motivations we take the domain $\Phi_n$ of  $v_n$ to consist of $\phi:[1,\tau_{n-m}]\times\bbT_n\to\bbC$ which are $C^2$ in
$t$ and $C^4$ in $x$ with  $ \partial^i_t\phi(1,x)=0$ for $0\leq i\leq 2$ and all $x\in \bbT_n$.
We equip $ \Phi_n$
 with
the sup norm  
$$\|\phi\|_{\Phi_n}:=\sum_{i\leq 2, |\al|\leq 4}\|\partial_t^i\partial_x^\al\phi\|_\infty.
$$ 

We will now set up the RG map \eqref{eq: vn+1new1} in a suitable
space of $v_n$, $v_{n-1}$ defined on   $\Phi_n$  and $\Phi_{n-1}$respectively. 

First, note that
for $\phi\in\Phi_{n-1}$, $s\phi(t,x)=\la^\hf\phi(\la^2t,\la x)$ is defined on $[\la^{-2},\tau_{n-m}]\times \bbT_{n}$. We extend
it to $[1,\tau_{n-m}]\times \bbT_{n}$ by setting $s\phi(t,x)=0$ for $1\leq t\leq \la^{-2}$. Next,
$\Gamma\xi_{n-1}$ vanishes (a.s.) on $[0,\la^2]$ and thus $s\Gamma\xi_{n-1}$ vanishes on
$[0,1]$. 
We can now state the final condition for the set $\caA_m$: for all $n> m$ we demand
\beq
 \|s\Gamma\xi_{n-1}\|_{\Phi_{n}}\leq \la^{-\ga n}.
 \label{eq: Gammaxibound}
%\non
\eeq 
Let $B_n\subset\Phi_n $ be the open ball centered at origin of radius $r_n=\la^{-2\ga n}$ and  $\caW_n(B_n)$ be the space of analytic functions from  $B_n$ to $\caV_n$ equipped with the supremum norm which we denote by
$\|\cdot\|_{B_n}$ (see \cite{bgk} for a summary of basic facts on analytic functions on Banach spaces). We will solve the fixed point problem  \eqref{eq: vn+1new1} in this space. We collect some elementary properties of these norms in
the following   lemma, proven in Section \ref{se:kernel estimates}:

\begin{lemma}\label{lem: gammamap}  

(a) $s\Gamma %$ is a bounded operator from $
: \caV_{n-1}\to \Phi_n$ and  $h_{n-1-m}\Gamma:\caV_{n-1}\to\caV_{n-1}$
are bounded operators with norms bounded by $C(\la)$. Moreover $s\Gamma h_{n-1-m} v= s\Gamma v$ as elements of $\Phi_n$.
\vskip 2mm
\noindent (b) $s: \Phi_{n-1}\to \Phi_n$ and $s^{-1}: \caV_{n}\to \caV_{n-1}$ are bounded with
$$ %\|s\Gamma\|\leq C(\la),\ \ \ 
\|s\|\leq \la^{\hf},\ \ \ \|s^{-1}\|\leq C\la^{-\hf}.$$ %uniformly in $n$.
%Moreover, if $\psi\in C^\infty(\bbR)$ with $\psi(t)=1$ for $t\geq 1$ then $s\Gamma\psi v= s\Gamma v$
%on $[1,\tau_{n-m}]$.
%$G_1  $ is a bounded operator from $\caV_m$  to $\Phi_m$ with norm $$\|G_1\|\leq C(\la).$$

\vskip 2mm
\noindent (c) Let $\phi %\in\Phi_n$ with $ \partial^i_t\phi|_{t=\tau_{n-m}}=0$ for $i\leq 2$
\in C^{2,4}(\bbR\times\bbT_n)$ and  $v\in\caV_n$. Then $\phi v\in\caV_n$ and $\|\phi v\|_{\caV_n}\leq C\|\phi\|_{C^{2,4}}\|v\|_{\caV_n}$. 
\vskip 2mm
%\noindent (c) Let  $v\in \caV_{n-1}$. Then $s\Gamma h_{n-1-m} v= s\Gamma v$. We also have $\|h_{n-1-m}\Gamma v\|_{ \caV_{n-1}}\leq C(\la)\| v\|_{ \caV_{n-1}}$.
\end{lemma}

 The linear RG \eqref{eq:linearrg} is controlled by 

\begin{proposition}\label{prop: linrgnorm} Given $\la<1$, $\ga>0$ there exists $n(\ga,\la) $ s.t. for $n\geq n(\ga,\la) $ $\caL_n$ maps $\caW_{n}(B_{n})$ into $\caW_{n-1}(\la^{-\hf}B_{n-1})$ %$\caW_n(\la^\hf r+C\la^{-\ga n})$  
with norm $\|\caL_n\|\leq C\la^{-\frac{5}{2}}$.
\end{proposition} 

\begin{proof} 
Let $v\in\caW_{n}(B_{n})$ and $\phi\in \la^{-\hf}B_{n-1}$. By \eqref{eq: Gammaxibound} and Lemma \ref{lem: gammamap} (a,b) 
%$$
%\|s\Gamma\xi_{n-1}\|_{\Phi_n}\leq C(\la)\la^{-\ga n}
%$$ 
%and so by 
%Recall the definition of $\caL_n$ in eq. \eqref{eq:linearrg}.
%\beq
%(\caL_nv)(\phi):=\la^{-2}s^{-1}v(s(\phi+\Gamma\xi_{n-1}))
%\non%
%\label{eq:linearrg11}
%\eeq
%where $(s\phi)(t,x)=\la^\hf\phi(\la^2t,\la x)$. We have $s:\Phi_{n-1}\to\Phi_n$ with
%$\|s\phi\|_{\Phi_n}\leq \la^\hf\|\phi\|_{\Phi_{n-1}}$ since $\la< 1$. 
%In \eqref{eq:linearrg}
%$(\Gamma\xi_{n-1})(t)$ enters there with $t\in[0,\la^{-2(n-1-m)}]$. From \eqref{eq:Gala } we then
%infer that $(\Gamma\xi_{n-1})(t)$ involves $\xi_{n-1}$ on the interval $[0,\la^{-2(n-1-m)}-\la^2]$.
%Hence we may replace $\Gamma\xi_{n-1}$ by $\Gamma h_{n-1-m}\xi_{n-1}$ and then apply Lemma \ref{lem: gammamap} (a)  to conclude that
%on $\caA_m$ we have 
%$$
% \|\Gamma\xi_{n-1}\|_{\Phi_{n-1}}\leq C(\la)\la^{-\ga (n-1)}.
% $$ 
% Lemma \ref{lem: gammamap} %(a) we then conclude%$\|s\phi\|_{\Phi_n}\leq \la^\hf\|\phi\|_{\Phi_{n-1}}$ so 
\beq
\|s(\phi+\Gamma\xi_{n-1})\|_{\Phi_n}\leq \la^{-2\ga (n-1)}+\la^{-\ga n}\leq \la^{-2\ga n}\non
%\label{eq:lnanal}
\eeq
for $n\geq n(\ga,\la) $. Hence $v(s(\phi+\Gamma\xi_{n-1}))$ is defined and analytic in
$\phi\in \la^{-\hf}B_{n-1}$ i.e.
 $\caL_n$ maps $\caW_{n}(B_{n})$ into $\caW_{n-1}(\la^{-\hf}B_{n-1})$
and by Lemma \ref{lem: gammamap}(b) 
\beq
\|\caL_nv\|_{\la^{-\hf}B_{n-1}}\leq C\la^{-\frac{5}{2}}\|v\|_{B_{n}}
\non
%\label{eq:lnanal}
\eeq
\end{proof}

As a corollary of Lemma \ref{lem: gammamap}(c) and \eqref{eq: amevent1} we obtain for $n\geq m$ and $N\geq n$ (recall \eqref{eq: undef} and \eqref{eq:U_n(0) }):
\beq
\|h_{n-m}u^{(N)}_n\|_{RB_n}\leq CR^3 \la^{(1-6\ga)n}%(\|\phi\|_{\Phi_n}+\la^{-\ga n})^3
\label{eq:unormbound }
\eeq
and
\beq
\|h_{n-m}(U^{(N)}_n(0)+DU^{(N)}_n(0)\phi)\|_{RB_n}\leq CR \la^{(2-3\ga)n}%(\|\phi\|_{\Phi_n}+\la^{-\ga n})^5%\log\la^n
\label{eq:Unormbound}
\eeq
for all $R\geq 1$ (since they are polynomials in $\phi$ with coefficients the noise fields $\zeta_n$).% In fact, we have

%\begin{lemma}\label{le: Vnbound} There exist $\la_0>0$, $\ga_0>0$ so that for $\la<\la_0$, $\ga<\ga_0$  and $m>m(\ga,\la) $
% if 
% $\ \Xi\in A_m\ $
%  then for all $N\geq n\geq m$ 
%\beq
 %\|%h_{n-m}
%h_{n-m}U_{n}^{(N)}\|_{B_{n}}\leq \la^{(2-11\ga)n}%A\la^{\ga(N-n)}\la^{(2-10\ga)n}. 
%\label{eq: VnNbound}
%\non
%\eeq
%\end{lemma}

Next we will rewrite the fixed point equation  \eqref{eq: vn+1new1} in a localized form. 
%Suppose $v^{(N)}_{n-1}\in $ and $v^{(N)}_n$ satisfy  \eqref{eq: vn+1new1}. 
Define $\tilde v^{(N)}_n=h_{n-m}v^{(N)}_n$ so that  
\beq
\tilde v^{(N)}_{n-1}(\phi)=h_{n-1-m}(\caL_n v^{(N)}_n)(\phi+\Gamma_{}\tilde v^{(N)}_{n-1}(\phi))% :=\caF_n(v^{(N)}_n,v^{(N)}_{n+1})
 \label{eq: vn+1new1tildeee}
%\label{eq: }
\eeq
where we used Lemma \ref{lem: gammamap}(a) in the argument. 
By \eqref{eq:linearrg} 
\beq
h_{n-m-1}\caL_n=\caL_nh_{n-m-1}(\la^2\cdot)
 \label{eq: h cal commu}
%\label{eq: }
\eeq
and $h_{n-1-m}(\la^2\cdot)$ is supported on  $[0,\tau_{n-m}-\hf]$. Since
 $h_{n-m}=1$ on $[0,\tau_{n-m}-\la^2]$ we get
 \beq
h_{n-m-1}(\la^2t)=h_{n-m-1}(\la^2t)h_{n-m}(t)
 \label{eq: h la commu}
%\label{eq: }
\eeq
so that \eqref{eq: vn+1new1tildeee} can be written as
 \beq
 \tilde v^{(N)}_{n-1}(\phi)=h_{n-1-m}(\caL_n \tilde v^{(N)}_n)(\phi+\Gamma_{} \tilde v^{(N)}_{n-1}(\phi)).
\non
%\label{eq: }
\eeq
Hence the $\nu$ fixed point problem \eqref{eq: wfp} becomes
\beq
\tilde\nu_{n-1}= h_{n-1-m}(\caL_n\tilde\nu_{n} +\tilde\caF_n(\tilde U_{n-1}+ \tilde\nu_{n-1})),\ \ \ \tilde \nu_N=0\label{eq: wfptilde}
\eeq
where $\tilde\caF_n$ is as in \eqref{eq:Fnequa } i.e.
\beq
\tilde\caF_n(w)=\caG_n(\tilde u_n+\tilde w_n,\tilde u_{n-1}+w)-D\tilde u_{n-1} \Gamma \tilde u_{n-1}%\caG(u_n,u_{n-1}))%+\tilde\caG_n.%(u_n,u_{n-1}).   
\label{eq:Fnequatilde }
%\non
\eeq

%\beq
%\tilde v^{(N)}_{n-1}(\phi)=h_{n-1-m}(\caL_nv^{(N)}_n)(\phi+\Gamma_{}\tilde v^{(N)}_{n-1}(\phi))% :=\caF_n(v^{(N)}_n,v^{(N)}_{n+1})
% \label{eq: vn+1new1tilde}
%\non%\label{eq: }
%\%eeq
%By \eqref{eq:linearrg} $\caL_n\tilde v^{(N)}_n=h_{n-m}(t/\la^2)\caL_n v^{(N)}_n$ and since$h_{n-m}(t/\la^2)=1$ on  $[0,\tau_{n-m-1}-\la^2]$  which contains the support$[0,\tau_{n-m-1}-\hf]$  of $h_{n-1-m}(t)$ we get
 Thus with only a slight abuse of notation we will drop the tildes and $h$ factors in the norms in the following:

\begin{proposition}\label{prop: solution of fp}  There exist $\la_0>0$, $\ga_0>0$ so that for $\la<\la_0$, $\ga<\ga_0$  and $m>m(\ga,\la) $
 if 
 $\ \Xi\in A_m\ $
  then  then for all $N\geq n-1\geq m$ 
 the equation  \eqref{eq: wfptilde} has a  unique solution $\nu_{n-1}^{(N)}\in \caW_{n}(B_{n-1})$. 
 These solutions satisfy
 \beq
 \|\nu_{n}^{(N)}\|_{B_{n}}\leq \la^{(3-\frac{1}{4})n}\label{eq: solball}
%\non
\eeq
and $\nu_n^{(N)}$ converge in $\caW_n(B_n)$ to a limit $\nu_n\in\caW_n(B_n)$ as $N\to\infty$.
 $\nu_n$ is independent on the small scale cutoff: $\nu_n=\nu'_n$.
\end{proposition} 

\begin{proof} 
We solve   \eqref{eq: wfptilde}  by Banach fixed point theorem in the ball % \eqref{eq: solball}
$ \|\nu_{n-1}\|_{B'}\leq \la^{(3-\frac{1}{4})(n-1)}$ where $B'=\la^{-\hf} B_{n-1}$ (we only need to prove analyticity in $B_{n-1}$ but for bounding $U_n$  the larger region is needed). By  %Lemma \ref{lem: gammamap}(b) and
Proposition \ref{prop: linrgnorm}  we have
\beq
 \| %h_{n-1-m}
 \caL_n\nu_{n} \|_{B'}\leq  C\la^{-\frac{5}{2}}\la^{(3-\frac{1}{4})n}=C\la^{\frac{1}{4}}
 \la^{(3-\frac{1}{4})(n-1)}
 \label{eq: calnubound}
%\non
\eeq
Next we estimate the $\caF_n$ term in \eqref{eq: wfptilde}.  $\caF_n$ is given in \eqref{eq:Fnequatilde } so we need to start with $\caG_n$ given in \eqref{eq: Gdef1}. Let  $v\in\caW_n(B_n)$ and $\bar v\in \caW_{n-1}(B')$ and
define
\beq
f(v,\bar v)(\phi):= \la^{-5/2}s^{-1}v(s(\phi+\Gamma\xi_{n-1}+\Gamma_{}\bar v(\phi)).%\caG_n(v, z\bar v)(\phi)=(\caL_nv)(\phi+z\Gamma_{}\bar v(\phi))-(\caL_nv)(\phi)
\label{eq: fdefinition}
%\non
\eeq
For $\phi\in B'$ we have by Lemma \ref{lem: gammamap}(a)  and \eqref{eq: Gammaxibound}
\beq
\|s(\phi+\Gamma\xi_{n-1}+\Gamma_{}\bar v(\phi))\|_{\Phi_n}\leq %\la^\hf
\la^{-2\ga (n-1)}+\la^{-\ga n}
+ C(\la)\|\bar v\|_{B'}\non
%\label{eq:lnanal}
\eeq
Hence $f(v,\bar v)\in\caW_{n-1}(B') $ provided
\beq
\|\bar v\|_{B'}\leq c(\la)\la^{-2\ga n}
%|z|\leq c(\la)\la^{-2\ga n}\|\bar v(\phi)\|_{\caV_{n-1}}^{-1}.
\label{eq:fanalregion}
\eeq
We use this first to estimate 
$$g_n:=\caG(u_n,u_{n-1})-Du_{n-1} \Gamma u_{n-1}.
$$
We have $g_n=f(1)-f(0)-f'(0)$ where $f(z)=f(u_n,zu_{n-1})$. By \eqref{eq:fanalregion} $f$ is analytic in
$$|z|<c(\la)\la^{-2\ga n}\|u_{n-1}\|_{B'}^{-1}
$$
 and so by 
a Cauchy estimate and \eqref{eq:unormbound } 
\beq
\|g_n\|_{B'} %C(\la)\|\bar v(\phi)\|^2_{\caV_{n-1}} \la^{4\ga(n-1)}\|\caL_nv\|_{B_{n-1}}
\leq C(\la) \la^{4\ga n}\| u_n\|_{B_n}\|u_{n-1}\|^2_{B'}\leq C(\la) \la^{(3-14\ga)n}
%\label{eq: Gdef1new}
\label{eq: geen}
%\non
\eeq
Next we write 
\beq
\caF_n(w)=g_n+\caG_n(w_n,u_{n-1}+w)+h_n(w)
%\caG_n( u_n+ w_n, u_{n-1}+w)-D\tilde u_{n-1} \Gamma \tilde u_{n-1}%\caG(u_n,u_{n-1}))%+\tilde\caG_n.%(u_n,u_{n-1}).   
\label{eq:Fnequatilde1 }
%\non
\eeq
with
\beq
h_n(w)=\caG_n(u_n,u_{n-1}+w)-\caG(u_n,u_{n-1})
%\label{eq:Fnequatilde1 }
\non
\eeq
We have $h_n(w)=\tilde f(1)-\tilde f(0)$ with $\tilde f(z)=f(u_n,u_{n-1}+zw)$ which is analytic in
$$|z|<c(\la)\la^{-2\ga n}\|w\|_{B'}^{-1}.
$$
Hence by a Cauchy estimate
\beq
\|h_n(w)\|_{B'} \leq  C(\la) \la^{2\ga n} \|u_n \|_{B_n}\|w\|_{B'}.
\label{eq: heen}
%\non
\eeq
Finally in the same way
\beq
\|\caG_n(w_n,u_{n-1}+w)\|_{B'} \leq C(\la) \|w_n \|_{B_n}(\|u_{n-1}\|_{B'}+\|w\|_{B'}).
\label{eq:caGnbound2 }
\eeq
Recalling  \eqref{eq:vndeco } we have
\beq
 \|w_n \|_{B_n}\leq  \|U_n \|_{B_n}+ \|\nu_n \|_{B_n}
\label{wnbound}
\eeq
so to proceed we need a bound for $U_n$. It is defined iteratively in  \eqref{eq:Uiteration and ic } which we write as
\beq
U_{n-1}=\caL U_n+\tilde U_{n-1}
\non
\eeq
where $\tilde U_{n-1}=Du_{n-1} \Gamma u_{n-1}=f'(0)$ where  $f(z)=f(u_n,zu_{n-1})$ as above.
 Again by Cauchy we get
\beq
\|\tilde U_{n-1}\|_{B'}\leq C(\la)\la^{2(1-6\ga)n+2\ga n}.
\label{tildeun-1}
\eeq
Recall the definition of $V_n$ in \eqref{eq: U_n}. It satisfies 
\beq
V_{n-1}(\phi)=(\caL_n V_n)(\phi)-(\caL_n V_n)(0)-D(\caL_n
V_n)(0)\phi+\tilde V_{n-1}(\phi).
\label{eq: Viteration11}
\non
\eeq
where
$\tilde V_{n-1}=\tilde U_{n-1}-\tilde U_{n-1}(0)-D\tilde U_{n-1}(0)\phi$. From \eqref{tildeun-1} we get
\beq
\|\tilde V_{n-1}\|_{B'}\leq C(\la)\la^{(2-10\ga)n}.
\label{tildeVn-1}
\eeq
Assume inductively
\beq
\|V_{n}\|_{B_n}\leq \la^{(2-11\ga)n}.
\label{Vnite}
\eeq
Proposition \ref{prop: linrgnorm} combined with a Cauchy estimate (here we use $B'=\la^{-\hf}B_{n-1}$) and \eqref{tildeVn-1} gives
\beq
\| V_{n-1}\|_{B_{n-1}}\leq C\la^{-3/2}\| V_{n}\|_{B_{n}}+C(\la)\la^{(2-10\ga)n}
\label{Vniteee}
\eeq
which proves the induction step taking $\gamma$ small enough and $n\geq n(\la)$. Since $U_N$ is linear by
\eqref{eq:Uiteration and ic } the induction starts with $V_N=0$. 
Combining \eqref{Vnite} with \eqref{eq:Unormbound} and the initial condition in \eqref{eq:Uiteration and ic }
we then arrive at
$$
 \|U_n \|_{B_n}\leq 2\la^{(2-11\ga)n}.
$$
Combining this bound with  \eqref{eq: geen}, \eqref{eq: heen} and \eqref{eq:caGnbound2 } gives, for
$\gamma$ small enough
\beq
\|\caF_n( U_{n-1}+ \nu_{n-1})\|_{B_{n-1}}
\leq \la^{(3-\frac{1}{4})n}+\la^{\hf n}(\|\nu_{n-1}\|_{B_{n-1}}+
\|\nu_{n}\|_{B_n})+\|\nu_{n-1}\|_{B_{n-1}}
\|\nu_{n}\|_{B_{n}}
\non
\eeq
Recalling  \eqref{eq: calnubound} and    Lemma \ref{lem: gammamap}(c) to bound the  $h_{n-1-m}$ factor in \eqref{eq: wfptilde} we conclude that  the ball 
$ \|\nu_{n-1}\|_{B_{n-1}}\leq \la^{(3-\frac{1}{4})(n-1)}$ is mapped by the RHS of   \eqref{eq: wfptilde}
to itself. The map is also a contraction if $n\geq n(\la)$ since by \eqref{eq: calnubound}  $\caL$ is and
\beq
\|\caF_n(U_{n-1}+\nu_1)-\caF_n(U_{n-1}+\nu_2)\|_{B_{n-1}}\leq  C(\la) \la^{(1-6\ga)n}
\|\nu_1-\nu_2\|_{B_{n-1}}.
 %\label{eq: }
\non
\eeq

Let us address the convergence as $N\to\infty$ and cutoff dependence of $\nu_n=\nu_n^{(N)}$. Recall we can deal with both questions together with $\nu'_n$. Since $u_n$ and $U_n-V_n$ are
polynomials in $\phi$ with coefficients $\zeta_n$ satisfying 
the estimate \eqref{eq: amevent2} we get
\baq
\|u_n-u'_n\|_{B_n}&\leq& C \la^{\ga (N-n)} \la^{(1-6\ga)n}%\|\chi-\chi'\|_\infty%(\|\phi\|_{\Phi_n}+\la^{-\ga n})^3
\label{eq:deltaunormbound1 }\\
\| (U_n-V_n)-(U'_n-V'_n)\|_{B_n}&\leq& C \la^{\ga (N-n)} \la^{(2-3\ga)n}.%\|\chi-\chi'\|_\infty.%(\|\phi\|_{\Phi_n}+\la^{-\ga n})^5%\log\la^n
\label{eq:deltaUnormbound1}
\eaq
To study $\caV_n:=V_n-V'_n$ we need to estimate (recall \eqref{eq: Viteration11})
$\tilde V_n-\tilde V'_n$ which in turn is determined by $\tilde U_n-\tilde U'_n$. This is again estimated by Cauchy and we get%=f(1)-f(0)$ where$f(z)=f(u'_n+z(u_n-u'_n),u'_{n-1}+z(u_{n-1}-u'_{n-1}))$. By a Cauchy estimate we obtain
$$\|\tilde U_n-\tilde U'_n\|_{B'}\leq C(\la) \la^{\ga (N-n)} \la^{(2-10\ga)n}.
$$
Proceeding as in the derivation of \eqref{Vniteee} we get
\beq
\| \caV_{n-1}\|_{B_{n-1}}\leq C\la^{-3/2}\| \caV_{n}\|_{B_{n}}+C(\la) \la^{\ga (N-n)} \la^{(2-10\ga)n}
\non%\label{Vniteee}
\eeq
leading to
\beq
\| \caV_{n}\|_{B_{n}}\leq C(\la) \la^{\ga (N-n)} \la^{(2-10\ga)n}.
\non%\label{Vniteee}
\eeq
Combining this with \eqref{eq:deltaUnormbound1} we arrive at
\beq
\| U_{n}-U'_n\|_{B_{n}}\leq  \la^{\ga (N-n)} \la^{(2-11\ga)n}.
\label{eq:deltaunormbound3 }
\eeq
Finally using \eqref{eq:deltaunormbound1 } and \eqref{eq:deltaunormbound3 } it is now straightforward to prove, for $\ga$ suitably small,
\beq
%\|\caF^{(N+1)}_n(U^{(N+1)}_{n-1}+ \nu^{(N+1)}_{n-1})-\caF^{(N)}_n(U^{(N)}_{n-1}+ \nu^{(N)}_{n-1})\|_{B_{n-1}}
\|\caF_n-\caF'_n\|_{B_{n-1}}
\leq  \la^{\ga(N-n)} \la^{(3-\frac{1}{4})n}+ \la^{\hf n}\|\nu_{n-1}-\nu'_{n-1}\|_{B_{n-1}}.
 %\label{eq: }
\non
\eeq
As in \eqref{eq: calnubound} we get
\beq
 \| \caL_n(\nu_{n}-\nu'_{n}) \|_{B_{n-1}}\leq  C\la^{-\frac{5}{2}} \| \nu_{n}-\nu'_{n} \|_{B_{n}}
  \label{eq: calnubound1}
%\non
\eeq
Hence for small $\ga$ we obtain inductively for $m\leq n\leq N$
\beq
 \| \nu_{n}-\nu'_{n} \|_{B_{n}}\leq  C\la^{\ga(N-n)}\la^{(3-\frac{1}{4})n}.
 % \label{eq: calnubound1}
\non
\eeq
This establishes the convergence of $ \nu^{(N)}_{n}$ to a limit that is independent on the 
short time cutoff.
 \end{proof}
 \begin{remark}\label{rem: cutoff1} Let us briefly indicate how the  cutoff \eqref{eq: regpdespace } can be accommodated to our scheme. We only need to modify the first RG step. For $n=N$ in \eqref{eq: vn+1new} the noise is replaced by the spatially smooth noise $\tilde\xi_{N-1}=\rho\ast\xi_{n-1}$ and 
 $\Gamma$ by $\tilde\Gamma$ where we use the cutoff $\chi(t-s)$ in \eqref{eq:Gala }.  $\tilde\Gamma$  is not infinitely smoothing but $s\tilde\Gamma v_{N-1}^N\in B_N$ nevertheless since at this scale
$v_{N-1}^N$ is as smooth as $\phi$ is.   

\end{remark}

%:
\section{Proof of Theorem \ref{main result} } \label{se:main result}

We are now ready to construct the solution $\phi^{(\ep)} $ of the $\ep$ cutoff equation
\eqref{eq: regpde }.  Recall that  formally $\phi^{(\ep)} $ is given on time interval
$[0,\la^{2m}]$ by eq. \eqref{eq: finalsolution} (with $n=m$) with $\phi_m$ given as the solution of eq. 
\eqref{eq:neweq1 } on  time interval
$[0,1]$. Hence we first need to study the $f$ iteration eq.  \eqref{eq: fn+1new}.
This is very similar to the $v$ iteration \eqref{eq: vn+1new} except there is no fixed point problem to be solved and there is no multiplicative $\la^{-2}$ factor. As  in \eqref{eq: vn+1new1tilde} for $v^{(N)}_{n}$ we study instead of  \eqref{eq: fn+1new} the localized iteration
\beq
\tilde f^{(N)}_{n-1}(\phi)=h_{n-1-m} %(\caL_n
s^{-1}\tilde f^{(N)}_n(s(\phi+\Gamma_{}(\tilde v^{(N)}_{n-1}(\phi)+\xi_{n-1})))% :=\caF_n(v^{(N)}_n,v^{(N)}_{n+1})
 \label{eq: fn+1new1}
%\label{eq: }
\eeq
for
$\tilde f_{n}^{(N)}=h_{n-m}f_{n}^{(N)}$. The following Proposition is immediate :

\begin{proposition}\label{prop: solution of fiteration} 
 Let $\tilde\nu_{n}^{(N)}\in \caW_{n}(B_n)$, $m\leq n\leq N$ be as in 
Proposition \ref{prop: solution of fp}. Then for $m\leq n\leq N$ $\tilde f_{n}^{(N)}\in \caW_{n}(B_n)$ and 
\beq
 f_{n}^{(N)}(\phi)=\phi+\eta_{n}^{(N)} +g_{n}^{(N)}(\phi).\label{eq: fdecompss}
%\non
\eeq
with
\beq
 \|\tilde g_{n}^{(N)}\|_{B_{n}}\leq \la^{\frac{3}{4}n}.\label{eq: solballforf}
%\non
\eeq
and $g_n^{(N)}$ converge in $\caW_n(B_n)$ as $N\to\infty$ to a limit $\tilde g_n\in\caW_n(B_n)$ which is independent on the short time cutoff.
\end{proposition} 

\begin{proof} %As in \eqref{eq: vn+1new1tilde} we may write  \eqref{eq: fn+1new} as
%\beq
%\tilde f^{(N)}_{n-1}(\phi)=h_{n-1-m}\la^2(\caL_n\tilde f^{(N)}_n)(\phi+\Gamma_{}\tilde v^{(N)}_{n-1}(\phi))% :=\caF_n(v^{(N)}_n,v^{(N)}_{n+1})
 %\label{eq: fn+1new1}
%\label{eq: }
%\eeq
We have
$$
\tilde g_{n-1}^{(N)}(\phi)=h_{n-1-m}(\Gamma_{}\tilde v^{(N)}_{n-1}(\phi)+
s^{-1}\tilde g^{(N)}_n(s(\phi+\Gamma_{}(\tilde v^{(N)}_{n-1}(\phi)+\xi_{n-1})))
%\la^2(\caL_n 
$$ 
Since $ \| \tilde v^{(N)}_{n-1} \|_{B_{n-1}}\leq C\la^{(1-3\ga)(n-1)}$ Lemma  \ref{lem: gammamap}(b) implies
\beq
 \|\tilde g_{n-1}^{(N)} \|_{B_{n-1}}\leq  C(\la)\la^{(1-3\ga)(n-1)}+C\la^{-\frac{1}{2}}\la^{\frac{3}{4}n}\leq
 \la^{\frac{3}{4}(n-1)}
%  \label{eq: calnubound}
\non
\eeq
The convergence and cutoff independence follows from that of $v^{(N)}_{n}$.
\end{proof}

We need the following lemma: % for solving \eqref{eq:neweq1 }:
\begin{lemma}\label{lem: G1map}  
$G_1  $ is a bounded operator from $\caV_n$  to $\Phi_n$ and
 $G_1(h_{n-1-m}(\la^2\cdot) v)=G_1v $.
%with norm
%$$\|G_1\|\leq C(\la)\la^{-2(n-m)}.$$
\end{lemma}

%\begin{proposition}\label{prop: solution of phi_m eq}  With the assumptions of Proposition  \ref{prop: solution of fp} the equation  \eqref{eq:neweq1mod } has, for $n=m+1$, a
%unique solution $\phi_m^{(N)}\in B_m$ and as $N\to\infty$ $\phi_m^{(N)}\to\phi_m\in B_m$  which is independent on the short time cutoff.
%\end{proposition} 
%\begin{proof}
%Recall that  $\phi\in B_m$ is defined on time interval $[0,1]$. Thus $(G_1v)(t)$, $t\in [0,1]$, depends on $v$ on $[0,1-\la^2]$. Hence $G_1v=G_1h_0v$ and we get
%\beq
%\|G_1(v_m^{(N)}(\phi)+\xi_m)\|_{\Phi_m}=\|G_1 h_0(v_m^{(N)}(\phi)+\xi_m)\|_{\Phi_m}\leq
%C(\la)(\la^{(1-6\ga)m}+\la^{-\ga m})
 % \label{eq: calnubound1}
%\non
%\eeq
%Hence the RHS of  \eqref{eq:neweq1 } maps $B_m$ to itself if $m\geq m(\la,\ga)$ and
%is readily seen to be a contraction. Convergence and cutoff independence as $N\to\infty$ follows then from that of $v_m^{(N}$.

%\end{proof}

\noindent {\it Proof of Theorem \ref{main result}  }.
We claim that if  $\Xi\in\caA_m$ the solution $\varphi^{(N)}$ of equation \eqref{eq: regpde } with $\ep=\la^N$ is
given by  (recall \eqref{eq: finalsolution}) 
\beq
%s_{\ep_n}
\varphi^{(N)}=s^{-m}\tilde f^{(N)}_m(0)%\phi^{(N)}_m)
 \label{eq: finalsolution1}
%\non
\eeq 
on the time interval $[0,\hf \la^{-2m}]$.
%We proceed by induction in {\it increasing} $n$. 
Let $\phi_n\in\Phi_n$ be defined inductively by $\phi_m=0$ and for $n>m$
\beq
 \label{eq:phinplus1 }
\phi_{n}= s(\phi_{n-1}+\Ga(\tilde v^{(N)}_{n-1}(\phi_{n-1})+\xi_{n-1})).%\non
\eeq
We claim that for all $m\leq n\leq N$ $\phi_n\in B_n$ and
\beq
 \label{eq:iteration for solution }
\phi_{n}= G_1 (\tilde v^{(N)}_{n}(\phi_{n})+\xi_{n}).%\non
\eeq
Indeed, this holds trivially for $n=m$ since the RHS vanishes identically on $[0,1]$.
Suppose $\phi_{n-1}\in B_{n-1}$ satisfies
\beq
 \label{eq:neweq1mod }
\phi_{n-1}= G_1 (\tilde v^{(N)}_{n-1}(\phi_{n-1})+\xi_{n-1}).%\non
\eeq
Then, first by  by  Lemma  \ref{lem: gammamap}(b) 
\beq
 \non%\label{eq:iteration for solution }
\|\phi_{n}\|_{\Phi_{n}}\leq \la^\hf\|\phi_{n-1}\|_{\Phi_{n-1}}+C(\la)\la^{-\ga n}\leq \la^{-2\ga n}%\non
\eeq
so that 
 $\phi_{n}\in B_{n}$.  Second, we have
by
 \eqref{eq:neweq1mod } and \eqref{eq:phinplus1 }
\baq
 \label{eq:phinplus1' }
\phi_{n}&=& s((G_1+\Ga)(\tilde v^{(N)}_{n-1}(\phi_{n-1})+\xi_{n-1}))=G_1\la^2s(\tilde v^{(N)}_{n-1}(\phi_{n-1})+\xi_{n-1})\non\\
&=& G_1(h_{n-1-m}(\la^2\cdot)\tilde v^{(N)}_{n}(\phi_{n})+\xi_{n})=G_1(\tilde v^{(N)}_{n}(\phi_{n})+\xi_{n})
\eaq
where in the third equality we used the RG iteration \eqref{eq: vn+1new1tildeee} and in the last equality
Lemma \ref{lem: G1map}. 

From \eqref{eq: fn+1new1} we have since $\phi_m=0$
\beq
\tilde f^{(N)}_{m}(0)=h_{0} %(\caL_n
s^{-1}\tilde f^{(N)}_{m+1}
(\phi_{m+1})=h_{0} h_{1}(\cdot/\la^2) %(\caL_n
s^{-2}\tilde f^{(N)}_{m+2}(\phi_{m+2})=h_{0} 
s^{-2}\tilde f^{(N)}_{m+2}(\phi_{m+2})
%(s(\phi+\Gamma_{}(\tilde v^{(N)}_{n-1}(\phi)+\xi_{n-1})))% :=\caF_n(v^{(N)}_n,v^{(N)}_{n+1})
% \label{eq: fn+1new1}
%\label{eq: }
\non
\eeq
where we used \eqref{eq: h la commu}. Iterating we get
%\beq
%\tilde f^{(N)}_{n-1}(\phi_{n-1})=h_{n-1-m} %(\caL_n
%s^{-1}\tilde f^{(N)}_n
%(\phi_n)=h_{n-1-m} h_{n-m}(\cdot/\la^2) %(\caL_n
%s^{-2}\tilde f^{(N)}_{n+1}(\phi_{n+1})=\dots
%(s(\phi+\Gamma_{}(\tilde v^{(N)}_{n-1}(\phi)+\xi_{n-1})))% :=\caF_n(v^{(N)}_n,v^{(N)}_{n+1})
% \label{eq: fn+1new1}
%\label{eq: }
%\non
%\eeq
%Since $h_{n-1-m} $ is supported on $[0,\tau_{n-1-m}-\hf\la^2]$ and $ h_{n-m}(\cdot/\la^2)=1$
%on $[0,\tau_{n-1-m}-\la^4]$ their product equals  $h_{n-1-m} $ and we get since $\phi_m=0$
\beq
\tilde f^{(N)}_{m}(0)=h_{0} %(\caL_n
s^{-(N-m)}\tilde f^{(N)}_N=h_{0}h_{N-m}(\cdot/\la^{2(N-m)})s^{-(N-m)}\phi_N=h_{0}s^{-(N-m)}\phi_N
\label{eq:tildefn0 }
\eeq
again by \eqref{eq: h la commu}.
Now $\phi_N\in B_N$ solves \eqref{eq:iteration for solution } with
 $\tilde v^{(N)}_N(\phi)=h_{N-m} v^{(N)}_N(\phi)$ with $v^{(N)}_N$ given by \eqref{eq: first v}.
Since $h_{N-m} =1$ on $[0,\tau_{N-m}-\la^2]$ we obtain
\beq
\non
% \label{eq:iteration for solution }
\phi_{N}= G_1 ( \tilde v^{(N)}_{n}(\phi_{n})+\xi_{n})=G_1 ( v^{(N)}_{n}(\phi_{n})+\xi_{n}).%\non
\eeq
and thus $\varphi^{(N)}=s^{-N}\phi_{N}$ solves
\eqref{eq: finalsolution1} on the time interval $[0,\la^{-2m}]$. \eqref{eq:tildefn0 } then gives
\beq
\non
% \label{eq:iteration for solution }
h_0(\la^{-2m}\cdot)\varphi^{(N)}=s^{-m}\tilde f^{(N)}_{m}(0)%\non
\eeq
so that \eqref{eq: finalsolution1} holds on the time interval $[0,\hf \la^{-2m}]$.

By Proposition \ref{prop: solution of fiteration} %and \ref{prop: solution of phi_m eq} 
$f^{(N)}_m(%\phi^{(N)}_m
0)$ converges in $\caV_m$ to a limit $\psi_m$ which is independent on the
short distance cutoff. %cutoff function $\chi$ in  \eqref{eq: regpde }. That is: if we run the iteration with the cutoff \eqref{eq: gammaN-ndef} $\psi_m$ is independent of $\chi'$. It is obviously also independent of the $\chi$ that
%enters the inductive steps.
Convergence in  $\caV_m$ implies convergence in $\caD'([0,1]\times \bbT_m)$. 
The claim follows from continuity of $s^{-m}:\caD'([0,1]\times \bbT_m)\to \caD'([0,\la^{2m}]\times \bbT_1)$.
 \qed

%:
\section{Kernel estimates}\label{se:kernel estimates}

In this Section we prove Lemmas  \ref{lem: gammamap}, \ref{lem: G1map}  and  \ref{lem: heatkernel} and give bounds for the various kernels entering the  proof of Proposition \ref{prop: mainproba}.
%\vskip 2mm
\subsection{Proof of Lemma  \ref{lem: gammamap} and \ref{lem: G1map}  }

%\noindent {\it Proof of Lemma \ref{lem: gammamap}} 
{\it Lemma  \ref{lem: gammamap} (a) and Lemma  \ref{lem: G1map} }
Let
$v\in C_0^\infty(\bbR_+\times\bbT_{n-1})$. Then
\beq
  s\Ga v(t)= \int_0^\infty k(\la^2 t-s)e^{(t-s)\Delta} v(s)ds
  %\label{eq: gaagain}%
  \non
\eeq
where $k(\tau)=\la^\hf(\chi(\tau)-\chi(\tau/\la^2))$ vanishes for $\tau\leq \la^2$. Hence 
$ s\Ga v\in C_0^\infty([1,\infty)\times\bbT_n)$. Next,  write
 $v= (-\partial_t^2+1)(-\Delta+1)^2Kv$ so that setting
$w=Kv$ we have
 \beq
  s\Ga v(t)= \int_\bbR k(\la^2 t-s)(-\Delta+1)^2e^{(\la^2t-s)\Delta} (-\partial_s^2+1)w(s)ds
  \label{eq: gaagain}%\non
\eeq
Integrating by parts  we get
 \beq
  s\Ga v(t)= \int_\bbR ((-\partial_s^2+1)k(\la^2t-s)(-\Delta+1)^2e^{(\la^2t-s)\Delta}) w(s)ds
  .
  \non
\eeq
The  kernels $\partial^a_t(k(\la^2t-s)\partial^\al_xe^{(\la^2t-s)\Delta}(x-y))$ are smooth, exponentially decreasing in $|x-y|$ and  and supported on $\la^2t-s\in [\la^2,2]$ for all $a$ and $\al$. Hence
the corresponding operators $O_{a\al}$ satisfy
$$
|(O_{a\al}1_{c_i}w)(t,x)|\leq C(\la)e^{-cd(i,(t,x))}\|w\|_{L^2(c_i)}
$$
which in turn yields our claim
$
\|s\Ga v\|_{\Phi_n}\leq C(\la) \|v\|_{\caV_{n-1}}
$. Hence $s\Gamma$ extends to a bounded operator  from $\caV_{n-1}$  to $\Phi_n$.

$G_1v$  is given by \eqref{eq: gaagain} with $k$ is replaced by
$1-\chi(t-s)$. The time integral is confined to  $[0,\tau_{n-m}]$ so that $\|G_1 v\|_{\Phi_n}\leq C(\la,n) \|v\|_{\caV_{n}}$.

{\it Lemma  \ref{lem: gammamap} (b)}. Recall $(s\phi)(t,x)=\la^\hf\phi(\la^2t,\la x)$ on $[\la^{-2},\tau_{n-1-m}]\times\bbT_{n-1}$.  Since $\la< 1$ we then get $\|s\phi\|_{\Phi_n}\leq \la^\hf\|\phi\|_{\Phi_{n-1}}$. 
Next, let $v\in C_0^\infty(\bbR_+\times\bbT_{n})$ and $w=Kv$.
 First write
$$Ks^{-1}v=
Ks^{-1}(-\partial_t^2+1)(-\Delta+1)^2w=K_1(-\la^4\partial_t^2+1)K_2(-\la^2\Delta+1)^2s^{-1}w.
$$
Next, $K_1(-\la^4\partial_t^2+1)=\la^4+(1-\la^4)K_1$ and $$K_2(-\la^2\Delta+1)^2=\la^4+2\la^2(1-\la^2)(-\Delta+1)^{-1}+(1-\la^2)^2K_2.
$$
Thus we need to show the operators $K_i$, $(-\Delta+1)^{-1}$ and $\la^\hf s^{-1}$ are bounded
in the norm $\sup_i\|\cdot\|_{L^2(c_i)}$ uniformly in $\la$. For  $K_i$ this follows from
the bounds \eqref{eq: K_1def} and \eqref{eq: K2bounded} and for the third one  from
$
(-\Delta+1)^{-1}(x,y)\leq Ce^{-c|x-y|}|x-y]^{-1}
$. 
Finally, let $c_i/\la:=\{(t/\la^2,x/\la)|(t,x)\in c_i\}$
\beq
\|\la^\hf s^{-1} w\|_{L^2(c_i)}^2=\la^5\int_{c_i/\la}|w|^2\leq\la^5\sum_{c\cap (c_i/\la)\neq\emptyset}%
\|w\|
_{L^2(c)}^2\leq C %\label{eq: }
\non
\eeq

%\vskip 2mm

\noindent {\it Lemma  \ref{lem: gammamap} 
(c) } Let $\phi\in C^{2,4}(\bbR\times\bbT_n)$ and $v\in C_0^\infty(\bbR_+\times\bbT_n)$ and set again $w=Kv$ so that 
\beq
\phi v=\phi(-\partial_t^2+1)(-\Delta+1)^2w%
%(-\partial_t^2+1)(-\Delta+1)^2(\phi w)+(-\partial_t^2+1)\sum_{|\al|\leq 3}
%\partial_x^\al(w\partial_x^{\be}\phi)%
%\label{eq: tocommute}
\non
\eeq
Using $\phi\partial_t^2f=\partial_t^2(\phi f)-2\partial_t(\partial_t\phi f)+\partial_t^2\phi f$ and similar commutings for $\Delta$ we get
\beq
K(\phi v)=%\phi w+2\partial_tK_1(\partial_t\phi w)-K_1(\partial_t^2 \phi w).
\phi w+\sum_a\caO_a(\phi_aw)
\non
\eeq
where the operators $\caO_a$ belong to the set $\{\partial_t^nK_1, \partial^\al_xK_2,\partial_t^nK_1\partial^\al_xK_2\}$ with $n\leq 1$ and $|\al|\leq 3$. The functions $\phi_a$ are multiples of $\partial^m_t\partial_x^\beta\phi$ with $m\leq 2$ and $|\beta|\leq 4$ and hence bounded in sup norm by $C\|\phi\|_{C^{2,4}}$. 
We get 
$$
\|\phi v\|_{\caV_n}\leq C\|\phi\|_{C^{2,4}}\max_a\sup_i\sum_j\|\caO_a\|_{L^2(c_j)\to L^2(c_i)}
$$
The operators $\partial_tK_1$ and $K_1$ have bounded exponentially decaying kernels. The operators $\partial^\al_xK_2$ are bounded in $L^2$ hence from $L^2(A)\to L^2(B)$ with $A,B\subset\bbT_n$.
Moreover their kernels $\partial^\al_xK_2(x-y)$ are smooth for $x-y\neq 0$ and exponentially decaying.
We conclude 
$$
\|\caO_a\|_{L^2(c_j)\to L^2(c_i)}\leq Ce^{-c|i-j|}
$$
and then
$$
\|\phi v\|_{\caV_n}\leq C\|\phi\|_{C^{2,4}}\|v\|_{\caV_n}
.$$ 
\qed

\subsection{Proof of Lemma \ref{lem: heatkernel} } 
%\noindent {\it Proof of Lemma \ref{lem: heatkernel} } 

Let $
H(t,x)=e^{t\Delta}(0,x)=(4\pi s)^{-3/2}e^{-x^2/4t}
$
be the heat kernel on $\bbR^3$. Then
\beq
H_n(t,x)=\sum_{i\in\bbZ^3}H(t,x+\la^{-n}i).
\label{eq:heatkernelf }
\eeq
Denoting  $\ep=\la^{2(N-n)}$ and separating the $i=0$ term we have 
 \beq
\bbE \eta^{(N)}_{n}(t,x)^2=\int_0^t(8\pi s)^{-3/2}(\chi(s)^2-\chi(s/\ep^2)^2)ds+\al(t)%\caO(e^{-c\la^{-2n}})
\label{eq:eeeta }
%\non
\eeq 
where
$$|\al(t)|\leq \sum_{i\neq 0}\int_0^{2}(8\pi s)^{-3/2}e^{-i^2/(4s\la^{2n})}\leq  Ce^{-c\la^{-2n}}
.
$$
Let ${\al'}(t)$ have the lower cutoff replaced by $\chi'$. Then
$$
|\al(t)-{\al'}(t)\|\leq C%\int_{\la^2\ep}^{2\ep}
\int s^{-3/2}|\chi(s/\ep'^2)^2-\chi(s/\ep^2)^2|e^{-\la^{-2n}/4s}ds\leq Ce^{-c\la^{-2N}}\|\chi-\chi'\|_\infty.
$$
Denote  the first term in \eqref{eq:eeeta } by  $\beta(t,\ep)$ and $\beta'(t,\ep)$ where the lower cutoff is $\chi'$. Then
$$
\beta'(\infty,\ep)=\ep^{-1}\int_0^\infty(8\pi s)^{-3/2}(\chi(\ep^2s)^2-\chi'(s)^2)ds=\ep^{-1}\rho'-\rho
$$
%with $\rho'=\int_0^{\infty}(8\pi s)^{-3/2}(1-\chi'(s)^2)ds<\infty$. 
So
$$
\delta_n^{(N)}(t)=\al(t)+\rho-\ga(t,\ep)
$$
with 
$$
\ga(t,\ep)=\beta(\infty,\ep)-\beta(t,\ep)=\int_t^\infty(8\pi s)^{-3/2}(\chi(s)^2-\chi(s/\ep^2)^2)ds\leq Ct^{-\hf}.
$$
Moreover
$$
|\ga(t,\ep)-\ga'(t,\ep)|=\int_t^\infty(8\pi s)^{-3/2}|\chi(s/\ep^2)^2-\chi'(s/\ep^2)^2|ds\leq Ct^{-\hf}\|\chi-\chi'\|_\infty1_{[0,\ep]}(t).
$$
%Thus
%$$
%|\delta_n^{(N)}(t)-\delta'_n^{(N)}(t)|\leq|\al_n^N(t)-\al'_n^N(t)|+|\ga(t,\ep)-\ga'(t,\ep)|\leq
%C(t^{-\hf}1_{[0,\la^{2(N-n)}]}(t)+e^{-c\la^{-2N}})\|\chi-\chi'\|_\infty
%$$
\qed

\subsection{Covariance and response function bounds}

We prove now bounds for the  covariance and response kernels \eqref{eq: gamman-ndef} and  \eqref{eq:Cndef } that are needed for the probabilistic estimates in Section \ref{se:noise}.  These kernels are translation invariant in the spatial variable and we will denote $C^{(N)}_n(t',t, x',x)$ simply by
$C^{(N)}_n(t',t, x'-x)$ and similarly for the other kernels. As before primed kernels and fields have the lower cutoff 
 $\chi'$.
%Denote  by $\eta'^{(N)}_{n}$ the corresponding field \eqref{eq:etandefinition } and 
We need to introduce  the mixed covariance
\beq
C'^{(N)}_n(t',t, x'-x):= \bbE \eta'^{(N)}_{n}(t',x') \eta^{(N)}_{n}(t,x)%G_n(t'-t)\int_0^tG_n(2s)(\chi(s)-\chi(\la^{-2n}s))^2dsG_n(t'-t)
%e^{(t'-t)\Delta}
%\int_0^tH_n(t'-t+s,x'-x)\chi_{N'-n}(t'-t+s)\chi_{N-n}(s)ds:=C^{(N',N)}_n(t',t, x'-x)
\label{eq:CNN'def }
\eeq
Let us define 
\baq
 \caC_n(\tau,x)&=&\sup_{|t'-t|=\tau}\sup_{N\geq n} C'^{(N)}_n(t',t, x)
 %\label{eq: }
\non
\\
 \caC^N_n(\tau,x)&=&\sup_{|t'-t|=\tau} |C'^{(N)}_n(t',t, x)-C^{(N)}_n(t',t, x)|
 %\label{eq: }
\non
\\
\caG_n(t,x)&=&\sup_{N\geq n} \Gamma^N_n(t, x) 
 %\label{eq: }
\non
\\
\caG^N_n(t,x)&=&|\Gamma'^{N}_n(t, x) -\Gamma^N_n(t, x) |
%\label{eq: }
\non
\eaq

The regularity of these kernels is summarized in
 
\begin{lemma}\label{lem: greg}  $\caC_n\in L^p(\bbR\times\bbT_n)$   uniformly in $n$
for  $p<5$ 
and 
\beq
 \|\caC_n^N\|_p\leq C\la^{\ga_p(N-n)}\|\chi-\chi'\|_\infty.\label{eq: cacnnbound}
%\non
\eeq
for $\ga_p>0$ for $p<5$.
 $\caG_n\in L^p(\bbR\times\bbT_n)$   uniformly in $n$
for  $p<5/3$ 
and 
\beq
 \|\caG_n^N\|_p^p\leq C\la^{\ga'_p(N-n)}
 \|\chi-\chi'\|_\infty.\label{eq: cacnnbound}
%\non
\eeq
for $\ga'_p>0$ for $p<5/3$.
 
\end{lemma} 

\begin{proof} As in  \eqref{eq:Cndef } we have, for $t'\geq t$:
\beq
C'^{(N)}_n(t',t, x'-x)=\int_0^tH_n(|t'-t|+2s,x'-x)\chi^1_{N-n}(t'-t+s)\chi^2_{N-n}(s)ds
\label{eq:CNN'defi }
\eeq
where $\chi^1=\chi'$, $\chi^2=\chi$ or vice versa depending on $t'>t$ or  $t'<t$. The heat kernel $H_n$ is pointwise positive.  
 Using  $\chi^1_{N-n}(t'-t+s)\chi^2_{N-n}(s)\leq 1_{[0,2]}(s)1_{[0,2]}(t'-t)$ we
may bound
\beq
C'^{(N)}_n(t',t,x)\leq C_n(t'-t,x)
1_{[0,2]}(t'-t)
\label{eq:c'nnbound}
\eeq
where %$C(t'-t,x)$ is the kernel $\caC(0,x)$ of the operator 
\beq
 C_n(\tau,x)=\int_0^2 H_n(\tau+2s,x)ds%\label{eq: }
\non
\eeq
From \eqref{eq:heatkernelf } we get
%Recalling that $H_n$ is the heat kernel  on $\bbT_n$ we have
\beq 
H_n(\tau,x)=\sum_{m\in \bbZ^3}(4\pi t)^{-3/2}e^{-\frac{(x+\la^{-n}m)^2}{4t}}.%
%c(\tau,x+\epsilon_n^{-1}m)
\label{eq:persum1 }
%\non
\eeq
Therefore
\beq C_n(\tau,x)=\sum_{m\in \bbZ^3}%(4\pi t)^{-3/2}e^{-\frac{(x+\epsilon_n^{-1}m)^2}{4t}}%
c(\tau,x+\la^{-n}m)
\label{eq:persum }
%\non
\eeq
where  
\beq
 c(\tau,x)=\int_0^2(4\pi (\tau+2s))^{-3/2}e^{-\frac{x^2}{4(\tau+2s)}}ds\leq
  C(e^{-cx^2}(x^2+\tau)^{-\hf}1_{\tau\in[0,2]}+\tau^{-3/2}e^{-cx^2/\tau}1_{\tau>2})).\label{eq:ctaudeco }
%\non
\eeq
%Since $\tau\leq 2$ we may bound
%\beq
 %c(\tau,x)\leq Ce^{-cx^2}(x^2+\tau)^{-\hf}
%\label{eq: cbound}
%\non
%\eeq
Combining with  \eqref{eq:persum } and \eqref{eq:c'nnbound} (with $\chi'=\chi$)
\beq
\caC_n(\tau,x)\leq Ce^{-cx^2}(x^2+\tau)^{-\hf}1_{[0,2]}(\tau)
\label{eq: cbound}
%\non
\eeq
and so  $\caC_n\in L^p$ if $p<5$. 
To get \eqref{eq: cacnnbound} use
$$|\chi^1_{k}(t'-t+s)\chi^2_{k}(s)-\chi^1_k(t'-t+s)\chi^2_{k}(s)|\leq 1_{[0,2\la^{2k}]}(s)1_{[0,2]}(t'-t)
\|\chi-\chi'\|_\infty.$$
Hence 
\beq
\caC_n^N(\tau,x)\leq \sum_{m\in \bbZ^3}
 c_{N-n}(\tau,x+\la^{-n}m)1_{[0,2]}(\tau)\|\chi-\chi'\|_\infty %\label{eq: }
\non
\eeq
where
\beq
c_{M}(\tau,x):=\int_0^{2\la^{2M}}(4\pi (\tau+2s))^{-3/2}e^{-\frac{x^2}{4(\tau+2s)}}ds=\la^{-M}
c_{0}(\la^{-2M}\tau,\la^{-M}x).
 \label{eq:cMtaux }
%\non
\eeq
Hence using \eqref{eq:ctaudeco }
$$
\|c_{M}1_{[0,2]}\|_p^p=\la^{(5-p)M}\|c_{0}1_{[0,2\la^{-2M}]}\|_p^p\leq 
\la^{(5-p)M}(1+\int \tau^{\frac{3}{2}(1-p)}1_{[2,2\la^{-2M}]}dt)\leq
C\la^{\ga M}$$
with $\ga>0$ for $p<5$. This yields  \eqref{eq: cacnnbound}.

%To estimate $ \caS_n(t,x)$ we get first
In the same way
\beq
 \caG_n(t,x)\leq Ct^{-3/2}e^{-cx^2/t}1_{[0,2]}(t)
 \label{eq:caGbound }
%\non
\eeq 
which is in  $L^p$ for $p<5/3$. \eqref{eq: cacnnbound} follows then as above.
\end{proof}
We will later also need the properties of the kernel
\beq
S^{(N)}_n(t',t,x):= C^{(N)}_n(t',t,x)\Ga^{N}_n(t',t,x).\label{eq: Skernel}
%\non
\eeq
Set
\beq
 \caS_n(\tau,x):=\sup_{|t'-t|=\tau}\sup_{N\geq n} \caS^{(N)}_n(t',t, x)%\leq \caC_n(\tau,x) \caG_n(\tau,x)
\label{eq: calSdefin}
%\non
\eeq
We get from  \eqref{eq: cbound} and \eqref{eq:caGbound }
\beq
 \caS_n(t,x)\leq Ct^{-2}e^{-cx^2/t}1_{[0,2]}(t)\in L^p,\ \ \  p<5/4.
 \non
 \eeq
Finally 
let $S'^{N}_n$ have $\chi'$ in all the lower cutoffs in $C^{N}_n$ and $\Ga^{N}_n$ and set
\beq
\caS'^{(N)}_n(\tau,x)=\sup_{|t'-t|=\tau} |S'^{(N)}_n(t',t, x)-S^{(N)}_n(t',t, x)| .%\label{eq: }
\non
\eeq
Then
\beq
\| \caS_n^N\|_p^p\leq C\la^{\ga''_p(N-n)}\|\chi-\chi'\|_\infty.\label{eq: casnbou}
%\non
\eeq
for some $\ga''_p>0$ for $p<5/4$.

\begin{remark}\label{rem: other initial1}  Recall from Remark \ref{rem: other initial} that
or the initial condition \eqref{eq: regpdeini } the covariance of  $\eta^{(N)}_n$ is the stationary one
\eqref{eq:Cndefstat }. Hence Lemma \ref{lem: greg}  holds for it as well.\end{remark}

\section{Proof of Proposition \ref{prop: mainproba}}\label{se:noise}

In this section we prove Proposition \ref{prop: mainproba}.
%\begin{proposition}\label{pro: proba}
%Let $\zeta^{(N)}_n$ be any of the fields in \eqref{eq: noisefields}. There exists
%$b$ s.t. for all $R>0$ and $p<\infty$
%\beq
%\bbP(\exists N\geq n: \|\zeta^{(N)}_n\|_{\caV_n}\geq R)\leq C(p,\la)R^{-p}\la^{-5n}
%\label{eq: noisestzetai}
%\non
%\eeq
%The same estimate holds for $ \|s\Gamma \xi_{n-1}\|_{\Phi_n}$ as well (which is $N$-independent).
%\end{proposition} 
The strategy is straightforward. We need to compute the covariances of the various fields in \eqref{eq: noisefields} and establish
enough regularity for them. Covariance estimate is all we need since the probabilistic bounds are readily derived from it. The covariances have of course expressions in terms of Feynman diagrams only one of which is diverging as $N\to\infty$.  The renormalization constant $b$ is needed to cancel that divergence. We don't introduce the terminology of diagrams since the ones that enter are simple enough to be expressed without that notational device.
 
\subsection{Covariance bound}

 We will deduce Proposition \ref{prop: mainproba} from a covariance bound for  the fields in \eqref{eq: noisefields}.  Let $\zeta_n^{(N)}(t,x)$ or  $\zeta_n^{(N)}(t,x,s,y)$ be any of the fields in \eqref{eq: noisefields}. 
Let
  \beq
 \tilde K(t',t,x)=e^{\hf\dist(t',I)}K(t'-t,x)h_{n-m}(t)
 \label{eq:ktildedefi}
\eeq
where $I=[0,\la^{-2(n-m)}]$ and define 
 $$\rho^{(N)}_n=\tilde K \zeta_n^{(N)}\ \ \ {or}\ \ \ \rho^{(N)}_n=\tilde K\otimes \tilde K
 \zeta_n^{(N)}
.$$
Then
  \beq
\|K\tilde\zeta^{(N)}_n\|_{L^2(c_i)}\leq Ce^{-\hf \dist (i_0,I)}\|\rho^{(N)}_n\|_{L^2(c_i)}.
\label{eq:rhozeta}
\eeq
where $i_0$ is the time component of $i$ and similarly for the bi-local case.
We bound the covariance of $\rho^{(N)}_n$:

\begin{proposition}\label{pro: proba1}
There exist renormalization constants $m_1,m_2, m_3$   and $\ga>0$ s.t.  for all $0\leq n\leq N<\infty$
\baq
\bbE\rho_n^{(N)}(t,x)^2&\leq& C \label{eq: twopoint}\\
\bbE({\rho'}_n^{(N)}(t,x)-\rho_n^{(N)}(t,x))^2&\leq& C\la^{\ga (N-n)}\|\chi-\chi'\|_\infty
  \label{eq: twopoint1}\\
  \bbE\rho_n^{(N)}(t,x,s,y)^2&\leq& Ce^{-c(|t-s|+|x-y|)}  \label{eq: twopoint2}\\
\bbE({\rho'}_n^{(N)}(t,x,s,y)-\rho_n^{(N)}(t,x,s,y))^2&\leq& C\la^{\ga (N-n)}e^{-c(|t-s|+|x-y|)} \|\chi-\chi'\|_\infty
  \label{eq: twopoint3}
\eaq

\end{proposition} 
\noindent{\it Proof of Proposition} \ref{prop: mainproba}. 
$\rho_n^{(N)}(t,x)^2$ belongs to the
inhomogenous Wiener Chaos of bounded order $m$ (in fact $m\leq 10$). Thus we get for 
 all $p>1$  %there exits a constant $C_p$  s.t. 
 \beq
\bbE\rho^{(N)}_n(t,x)^{2p}\leq (2p-1)^{pm}(\bbE\rho^{(N)}_n(t,x)^2)^p
\label{eq: 2ppoint}
\eeq
(see  \cite{nualart},  page 62). Using H\"older, \eqref{eq: 2ppoint} and \eqref{eq: twopoint} in turn we deduce
\baq
\bbE(\|\rho^{(N)}_n\|_{L^2(c_i)}^{2p})&\leq&\bbE(\|\rho^{(N)}_n\|_{L^{2p}(c_i)}^{2p})
=\int_{c_i} \bbE(\rho^{(N)}_n(t,x))^{2p}\non\\
&\leq& C_p\int_{c_i} ((\bbE\rho^{(N)}_n(t,x))^2)^p
\leq C_p
\label{eq: 2ppoint4}
\eaq
and thus by \eqref{eq:rhozeta}
\beq
\bbE(\|K\tilde\zeta^{(N)}_n\|_{L^2(c_i)}^{2p})\leq C_pe^{-p\dist(i_0,I_n)}
%\label{eq: noisest1}
\non
\eeq
so that
\beq
\bbP(\|K\tilde\zeta^{(N)}_n\|_{L^2(c_i)}\geq R)\leq C_p(R^{-1}e^{-\hf \dist(i_0,I_n)})^{2 p}
\label{eq: noisest1}
\non
\eeq
and finally 
\beq
\bbP(\|\tilde\zeta_n^{(N)}\|_{\caV_n}\geq \la^{-\ga n})\leq \sum_iC_p(\la^{\ga n}e^{-\hf \dist(i_0,I_n)})^{2 p}
\leq
C_p\la^{2\ga pn}\la^{2m}\la^{-5n}.
\label{eq: noisestim2}
%\non
\eeq
 For the bi-local fields we proceed
in the same way. First as in \eqref{eq: 2ppoint4} we deduce
\beq
\bbE(\|\rho^{(N)}_n\|_{L^2(c_i\times c_j)}^{2p})\leq C_p\int_{c_i\times c_j} ((\bbE\rho^{(N)}_n(t,x,s,y)^2)^p
\leq C_p(Ce^{-c\dist(c_i,c_j)} )^{2p}
%\label{eq: 2ppoint4}
\non
\eeq
and then use exponential decay to control
\beq
\bbP(\|\tilde\zeta_n^{(N)}\|_{\caV_n}\geq \la^{-\ga n})\leq \sum_{i,j}C_p(\la^{\ga n}e^{-c( \dist(i_0,I_n)+\dist(c_i,c_j))})^{ 2p}
\leq
C_p\la^{2p\ga n}\la^{2m}\la^{-5n}.
%\label{eq: noisest2}
\non
\eeq

Next we turn to \eqref{eq: amevent2} which we recall we want to hold for all cutoff functions $\chi, \chi'$ with bounded $C^1$ norm. We proceed by a standard Kolmogorov continuity argument.

Let $f(\chi):=K\tilde\zeta^{(N)}_n$. Without loss we may consider $\chi$ in the ball $B_r$ of radius $r$  at origin in
$C^1[0,1]$. As above we conclude from  \eqref{eq: twopoint1} 
\beq
\bbE(\|f(\chi)-f(\chi')\|_{L^2(c_i)}^{2p})
\leq C_p(\ep \|\chi-\chi'\|_\infty e^{-\dist(i_0,I_n)})^p
\label{eq: ffprime}
\eeq
with $\ep=\la^{\ga(N-n)}$.
Let $\chi_n$, $n=1,2,\dots$  be the Fourier coefficients of $\chi$ in the basis $1$, $\sin 2\pi x$, $\cos  2\pi x$, $\sin  4\pi x$, $\cos  4\pi x,\dots$.
Hence  $|\chi_n|\leq Crn^{-1}$. Let $\chi^m_n= \chi_n$ for $n\leq m$ and $\chi^m_n= \chi'_n$ for $n> m$. Let $Q_N$ be the dyadic rationals in $[0,1]$. Then, for $\be\in (0,1)$
\baq
\|f(\chi)-f(\chi')\|_{L^2(c_i)}\leq\sum_{m=1}^\infty\|f(\chi^m)-f(\chi^{m+1})\|_{L^2(c_i)}%\non\\
%&\leq&
\leq\sum_{m=1}^\infty|\chi_m-\chi'_{m}|^\be\sum_{N=0}^\infty
2^{\beta N}\Delta_{N,m}
%\label{eq: ffprime}
\non
\eaq
where
\beq
\Delta_{N,m}=\sup_{t,t'\in Q_N, |t-t'|=2^{-N}}\|g_m(t)-g_m(t')\|_{L^2(c_i)}
%\label{eq: ffprime}
\non
\eeq
and $g_m(t)=f(\chi_1,\dots,\chi_{m-1},t,\chi'_{m+1},\dots)$. Using \eqref{eq: ffprime} for
$g_m$ we get 
\beq
\bbP(\Delta_{N,m}>R)\leq C2^N(R^{-2}\ep 2^{-N}e^{-\dist(i_0,I_n)})^p
%\label{eq: noisestimatt2}
\non
\eeq
and then, taking $\be<\hf$ and $2\be p>1$
\baq
&&\bbP(\sup_{\chi,\chi'\in B_r}\|f(\chi)-f(\chi')\|_{L^2(c_i)}>R)\leq \sum_{m,N}\bbP(\Delta_{N,m}>R|\chi_m-\chi'_{m}|^{-\be}2^{-\be N})\non\\
&& \leq \sum_{m,N}
C2^{N}(R^{-2}\ep 2^{-N(1-2\be)}r^\be e^{-\dist(i_0,I_n)}m^{-2\be})^p\leq C(R^{-2}\ep r^\be e^{-\dist(i_0,I_n)})^p
%\label{eq: noisestimatt2}
\non
\eaq
since $|\chi_m-\chi'_{m}|\leq Crm^{-1}$. We conclude then
\beq
\bbP(\sup_{\chi,\chi'\in B_r}\|\tilde{\zeta'}_n^{(N)}-\tilde\zeta_n^{(N)}\|_{\caV_n}\geq \la^{\hf\ga(N-n)} \la^{-\ga n})%\leq \sum_iC_p(R^{-1}e^{-\hf c\dist(c_i,\Lambda_n)})^{2 p}
\leq
C_pr^{\be p}\la^{p\ga(N-n)}\la^{(2\ga p-5) n}\la^{2m}
.
\label{eq: noisestimatt2}
%\non
\eeq

We still need to deal with the last condition on $\caA_m$ in \eqref{eq: Gammaxibound}.
 By \eqref{eq:Gala } $\zeta:=\Gamma\xi_{n-1}$ is a Gaussian field with covariance
$$
\bbE\zeta(t',x')\zeta(t,x)=%e^{(t'-t)\Delta}
\int_0^tH_n(t'-t+2s,x'-x)
%e^{2s\Delta}
\chi(t'-t+s)\chi(s)ds\label{eq:zetacov}
$$
where $\chi$ is smooth with support in $[\la^2,2]$.
$
\bbE\zeta(t',x')\zeta(t,y)$ is smooth, compactly supported in $t'-t$ and exponentially decaying
in $x-y$. We get then by standard Gaussian estimates \cite{boga}	for $a\leq 2,|\al|\leq 4$ 
$$
 \bbP(\|\partial_t^a\partial_x^\al\Gamma\xi_{n+1}\|_{L^\infty(c_i)}>r)\leq C(\la)e^{-c(\la)r^2}
$$
and thus 
\beq
 \bbP(\|s\Gamma\xi_{n-1}\|_{\Phi_{n}}> \la^{-2\ga n})\leq  C(\la)\la^{2m}\la^{-5n}e^{-c(\la)\la^{-4n}}\label{eq: gaxiest}
\eeq
Combining \eqref{eq: noisestim2}, \eqref{eq: noisestimatt2} and \eqref{eq: gaxiest}
with a  Borel-Cantelli argument gives the claims  \eqref{eq: amevent1} and
\eqref{eq: amevent2}.
\qed

\subsection{Normal ordering}

Since the fields   \eqref{eq: noisefields}  are polynomials in gaussian fields the computation of their covariances is straightforward albeit tedious. To organize the computation it is useful to
express   them  in terms
of "normal ordered" expressions in the field $\eta_n$.  This provides also a transparent way to see why
 the renormalization constant $b$ is needed. We suppress again the superscript ${}^{(N)}$ unless needed for clarity. 
 
Define the "normal ordered" random fields
\beq
%h_n^{(0)}=1, \ \ 
:\eta_n^{}:=\eta_n, \ \ :\eta_n^{2}:=\eta_n^2-\bbE\eta_n^2, \ \ :\eta_n^{3}:=\eta_n^3-3\eta_n\bbE\eta_n^2
\label{eq: hndef}
\eeq
and   (recall Lemma \ref{lem: heatkernel})
 \beq
\delta_n(t):=\bbE(\eta^{(N)}_n(t,x))^2-3\rho_{N-n}.
\label{eq: deltan}
\eeq
Then
\beq
\eta_n^2-\rho_{N-n}=:\eta_n^{2}:+\delta_n,\ \ \eta_n^3-3\rho_{N-n}\eta_n=:\eta_n^{3}:+3\delta_n\eta_n
 %\label{eq: noisefields}
\non
\eeq
In virtue of Lemma \ref{lem: heatkernel} $\delta_n$ terms will turn out to give negligible contribution.

These normal ordered fields have zero mean and    the following covariances:
\beq 
 \bbE :\eta_n^{2}:(t,x):\eta_n^{2}:(s,y)=2C_n(t,s,x,y)^2\label{eq: exp1}
%\non
\eeq
and 
\beq 
  \bbE :\eta_n^{3}:(t,x):\eta_n^{3}:(s,y)%h_n^{(3)}(t,x)h_n^{(3)}(s,y)
 =6C_n(t,s,x,y)^3\label{eq: exp2}.
\eeq
where $C_n$ is defined in \eqref{eq:Cndef }.

Next we process $\omega_n$, $\zeta_n$ and $z_n$: % By \eqref{eq: noisefields}
\baq
\omega_n&=&3:\eta_n^{2}:\Gamma^N_n:\eta_n^{3}:+(m_2\log\la^N+m_3)\eta_n+3\delta_n\Gamma^N_n:\eta_n^{3}:+9
:\eta_n^{2}:\Gamma^N_n\delta_n\eta_n+9\delta_n:\Gamma^N_n\delta_n\eta_n
%\label{eq:omegadefin }
\non
\\
\frz_n&=&6\eta_n\Gamma^N_n:\eta_n^3:+18\eta_n\Gamma^N_n\delta_n\eta_n
%\label{eq:zetadefin }
\non
\\
z_n&=&9
:\eta_n^2:\Gamma^N_n:\eta_n^2:+(m_2\log\la^N+m_3)+9\delta_n\Gamma^N_n:\eta_n^2:+9:\eta_n^2:\Gamma^N_n\delta_n
%\label{eq:zdefin }
\non
\eaq
(Recall that $\omega_n(t,x)$ and $\frz_n(t,x)$ are functions, whereas $z_n(t,x,s,y)$ is a kernel).
Consider first the terms not involving  $\delta_n$, call them $\tilde\omega_n$, $\tilde \frz_n$ and $\tilde z_n$. We normal order $\tilde\omega_n$ and $\tilde z_n$:
\baq
\tilde\omega_n&=&(18C^2_n\Gamma^N_n-m_2\log\la^N-m_3)\eta_n+18:\eta_nC_n\Gamma^N_n\eta_n^{2}:+3:\eta_n^{2}\Gamma^N_n\eta_n^{3}:
\label{eq: tildeomega}
\\
\tilde z_n&=&(18
C^2_n\Gamma^N_n-m_2\log\la^N-m_3)+36:\eta_nC_n\Gamma^N_n\eta_n:+9:\eta_n^2\Gamma^N_n\eta_n^2:
\label{eq: tildez}
\eaq
where the product means point wise multiplication of kernels, e.g. the operator $C^2_n\Gamma^N_n$ has the kernel
\beq
B_n(t,x,s,y):=C_n(t,x,s,y)^2\Gamma^N_n(t,x,s,y).
\label{eq:Bndef }
\eeq
For Proposition  \ref{pro: proba1} it suffices to bound separately the covariances of all the fields in
\eqref{eq: tildeomega} and \eqref{eq: tildez} as well as the $\delta_n$-dependent ones.

\subsection{Regular fields}

We will now prove the claims of Proposition \ref{pro: proba1} for the fields $\zeta_n$ not requiring
renormalization.

It will be convenient to denote the space time points
$(t,x)$ by the symbol $z$. We recall from \eqref{eq: K_1def} and \eqref{eq: K2bounded} the bounds for  the kernel $K(z'-z)$ of the operator $K$ which imply a similar bound for  $\tilde K$ of \eqref{eq:ktildedefi}: 
\beq
0\leq  \tilde K(z',z)\leq Ce^{-\hf |z'-z|}:=\caK(z'-z).
  \label{eq: Ktildenewbound}
\eeq
We start
with the local fields. Let  $\zeta_n(z)$ be one of them and
\beq
\bbE\zeta_n(z_1)\zeta_n(z_2):=
H_n(z_1,z_2)
\non%  \label{eq: kxiesti}
\eeq
Then ($H_n$ is non negative, see below)
\baq
\bbE\rho_n(z)^2&=&\int%_{\La_n\times\La_n} 
\tilde K(z,z_1)\tilde K(z,z_2)H_n(z_1,z_2)dz_1dz_2
\non\\%  \label{eq: twopointmore}
&\leq&C
 \int%_{(\bbR\times\bbT_n)^2}
  e^{-\hf( |z-z_1|+ |z-z_2|)}H_n(z_1,z_2)dz_1dz_2\leq C\|\caH_n\|_1%\non
  \label{eq: ronbound}
\eaq
where
\beq
\caH_n(z):=\sup_{z_1-z_2=z}H_n(z_1,z_2)
\non%  \label{eq: kxiesti}
\eeq
Since $H_n$ is given in terms of products of the covariances $C_n$ and $\Ga_n^N$ we 
get an upper bound for $\caH$ by replacing $C_n$ and $\Ga_n^N$ by the translation invariant upper bounds $\caC_n$ and $\caG_n^N$ defined and bounded  in  Lemma \ref{lem: greg}. Let us proceed case by case.

\vskip 2mm

\noindent (a) $\ \zeta_n=:\eta_n^i:$. Using \eqref{eq: exp1} and \eqref{eq: exp2} we have
$$
\caH_n(z)\leq \caC_n(z)^i
$$
so by Lemma \ref{lem: greg} $\|\caH_n\|_1\leq \|\caC_n\|_i^i<\infty$.
In the same way we get
\beq
\bbE(\rho'%^{(N')}
_n(z)-\rho%^{(N)}
_n(z))^2\leq
%C \int e^{-\hf( |z-z_1|+ |z-z_1|)}\caC_n(z_1-z_2)^{i-1}\caC^{(N)}_n(z_1-z_2)^{}dz_1dz_2\non\\
%&\leq &
C\|\caC_n\|^{i-1}_i\|\caC^{(N)}_n\|_i\leq C\la^{\frac{5-i}{i}(N-n)}\leq C\la^{\frac{2}{3}(N-n)}\non
 % \label{eq: twopoint2}
\eeq

\vskip 2mm

 \noindent (b)  $\ \zeta_n=:\eta_nC_n \Gamma^{(N)}_n\eta_n^2:$.
We have
\beq
\bbE  :\eta_n(z_1)\eta_n(z_2)^2::\eta_n^{}(z_3)\eta_n(z_4)^2:=
2C_n(z_1,z_3)C_n(z_2,z_4)^2+4C_n(z_1,z_4)C_n(z_2,z_3)C_n(z_2,z_4)
%\dots\eta_n(z_k): :\eta_n(z'_1)\dots\eta_n(z'_{k'}):= \delta_{kk'}\sum_{\pi\in S_k} \prod_{i}C_n(z_i,z'_{\pi(i)})%\label{eq: }
%\eqref{eq: cbound} 
\non
\eeq
and so
\baq
\bbE  \zeta_n(z_1) \zeta_n(z_3)&=&\int%_{\La_n\times\La_n} 
(2C_n(z_1,z_3)C_n(z_2,z_4)^2+4C_n(z_1,z_4)C_n(z_2,z_3)C_n(z_2,z_4))\nonumber\\
&\cdot&S(z_1,z_2) S(z_3,z_4)dz_2dz_4
%\dots\eta_n(z_k): :\eta_n(z'_1)\dots\eta_n(z'_{k'}):= \delta_{kk'}\sum_{\pi\in S_k} \prod_{i}C_n(z_i,z'_{\pi(i)})%
\label{eq: contractions} 
%\non
\eaq
where
\beq
S(z_1,z_2):=C_n(z_1,z_2) \Gamma_n(z_1,z_2)\leq\caS_n(z_{12}) 
%\eqref{eq: cbound} 
\non
\eeq
where we use the notation $z_{12}=z_1-z_2$ and $\caS_n$ is defined in \eqref{eq: calSdefin}.  
%\beq S(z_1,z_2) \leq C|t_1-t_2|^{-2}e^{-(x_1-x_2)^2/4(t_1-t_2)}1_{[0,2]}(t_1-t_2):=\caS(z_1-z_2).%\label{eq: }\non\eeq
%Calling the bound in  Lemma \ref{lem: greg} by
%\beq C_n(z_1,z_2)\leq Ce^{-c|x_1-x_2|}(|x_1-x_2|^2+|t_1-t_2|)^{-\hf}1_{[0,2]}(t_1-t_2):=\caC(z_1-z_2).%\label{eq: }\non\eeq
Thus
\baq
0\leq \bbE  \zeta_n(z_1) \zeta_n(z_3)&\leq&
2\caC_n(z_{13})\int%_{(\bbR\times\bbT_n)^2} 
\caS_n(z_{12}) \caS_n(z_{34})
\caC_n(z_{24})^2dz_2dz_4+\non\\
&&4\int%_{(\bbR\times\bbT_n)^2}
 \caS_n(z_{12}) \caS_n(z_{34})
\caC_n(z_{14})\caC_n(z_{23})\caC_n(z_{24})dz_2dz_4\non\\
&:=&\caA(z_{13})+\caB(z_{13})
\label{eq: zeta2bound} 
%\non
\eaq
so that 
\beq
\bbE\rho_n(z)^2\leq %\int_{(\bbR\times\bbT_n)^2} e^{-\hf( |z-z_1|+ |z-z_3|)}(\caA(z_1-z_3)+\caB(z_1-z_3) )dz_1dz_3\leq 
C(\|\caA\|_1+\|\caB\|_1).
 \label{eq: twopoint2}
\eeq
Since $\caA=\caC_n( \caS_n\ast\caS_n\ast\caC_n^2)$ we get by H\"older and Young's inequalities
\beq
 \|\caA\|_1\leq  \|\caC_n\|_p \|\caS_n\ast\caS_n\ast\caC_n^2\|_q\leq  \|\caC_n\|_p \|\caS_n\|_r^2 \|\caC_n^2\|_{s}%=  \|\caS\|_1^2 \|\caC\|_{p}^3
 %\label{eq: }
\non
\eeq
where $2+\frac{1}{q}=\frac{2}{r}+\frac{1}{s}$. We can take e.g. $p=3$, $q=3/2$, $r=1$ and $s=3/2$ and
by Lemma \ref{lem: greg} this is finite. 
As for $\caB$,  we write
\baq
\|\caB\|_1&=&4\int \caS_n(z-z_2) \caS_n(z_4)
\caC_n(z-z_4)\caC_n(z_2)\caC_n(z_{24})dzdz_2dz_4\non\\&=&
(\caS_n\ast(\caC_n(\caS_n\ast
\caC_n))\ast\caC_n)(0)\leq \|\caS_n\|_r\|\caC_n(\caS_n\ast
\caC_n)\|_p\|\caC_n\|_q
%\int  \caS(z_4)\caC_n(z_2)(\caC_n(\caS\ast\caC_n))(z_2-z_4))dz_2dz_4
 %\label{eq: }
\non
\eaq
by Young's inequality with $2=\frac{1}{r}+\frac{1}{p}+\frac{1}{q}$.
Since $\caC_n\in L^a$  for $a<5$ and $\caS_n\in L^b$  for $b<5/4$ we have $\caS_n\ast
\caC_n\in L^c$ for $c<\infty$ and so $\caC_n(\caS_n\ast
\caC_n)\in L^p$ for $p<5$. So we may take e.g. $p=q=2$ and $r=1$.

In the same way, using  \eqref{eq: cacnnbound} and \eqref{eq: casnbou} we obtain
\beq
\bbE(\rho'%^{(N')}
_n(z)-\rho%^{(N)}
_n(z))^2\leq C\la^{\ga(N-n)}\|\chi-\chi'\|_\infty%(\la^{\frac{2}{3}(N-n)}+\la^{\frac{1}{4}(N-n)})\non
 \eeq
 for some $\ga>0$.
 
 \vskip 2mm

 \noindent (c) 
$\zeta_n=:\eta_n^{2}\Gamma^N_n\eta_n^{3}:$. Now
\baq
&&\bbE  :\eta_n(z_1)^2\eta_n(z_2)^3::\eta_n(z_3)^{2}\eta_n(z_4)^3:=
12C_n(z_1,z_3)^2C_n(z_2,z_4)^3\label{eq: threeterms} \\
&&+36C_n(z_1,z_4)^2 C_n(z_2,z_3)^2C_n(z_2,z_4)
+36C_n(z_1,z_3)C_n(z_1,z_4)C_n(z_2,z_3)C_n(z_2,z_4)^2.
\non%\label{eq: threeterms} 
\eaq
The first two terms have the same topology (as Feynman diagrams!) as the $\caA$ and $\caB$ above and we call their
contributions with those names again. Thus
\beq
 \|\caA\|_1\leq  \|\caC_n^2\|_p \|\Ga_n^N\ast\Ga_n^N\ast\caC_n^3\|_q\leq  \|\caC_n^2\|_p \|\Ga_n^N\|_r^2 \|\caC_n^3\|_{s}%=  \|\caS\|_1^2 \|\caC_n\|_{p}^3
 %\label{eq: }
\non
\eeq
where $2+\frac{1}{q}=\frac{2}{r}+\frac{1}{s}$. Now $\Ga_n^N$ is in $L^r$ for $r<5/3$. So we may take for
instance $p=2$, $q=2$, $s=4/3$, $r=8/7$ so that $ \|\caA\|_1\leq C \|\caC_n\|_4^5\|\Ga_n^N\|_{8/7}^2$.
For $\caB$ we get 
$$
\|\caB\|_1=36
(\Ga_n^N\ast(\caC_n(\Ga_n^N\ast
\caC_n^2))\ast\caC_n^2)(0)\leq 36
\|\Ga_n^N\|_r\|\caC_n(\Ga_n^N\ast
\caC_n^2)\|_p\|\caC_n^2\|_q
$$
 with $2=\frac{1}{r}+\frac{1}{p}+\frac{1}{q}$.
Since $\caC_n\in L^a$  for $a<5$ and $\Ga_n^N\in L^b$  for $b<5/4$ we have by Young $\Ga_n^N\ast
\caC_n^2\in L^c$ for $c<5$ and so $\caC_n(\Ga_n^N\ast
\caC_n)\in L^p$ for $p<5/2$. So we may take e.g. $p=q=2$ and $r=1$.

Finally the last term in \eqref{eq: threeterms}, call it $\caD$, is bounded by
\baq
 \|\caD\|_1&\leq& 36\int\caC_n(z) \Ga_n^N(z-z_2) \Ga_n^N(z_4)
\caC_n(z-z_4)\caC_n(z_2)\caC_n(z_{24})^2dzdz_2dz_4\non\\
&:=&\int f(z_2,z_4)g(z_2,z_4)dz_2dz_4
 %\label{eq: }
\non
\eaq
where $f(z_2,z_4)=\int\caC_n(z) \Ga_n^N(z-z_2)
\caC_n(z-z_4)dz$. We have
\beq
\int f(z_2,z_4)dz_2=(\caC_n\ast\caC_n)(z_4)
 \label{eq: fl1}
%\non
\eeq
which is in $L^\infty$ since $\caC_n\in L^p$, $p<5$. Hence
$$
 \|\caD\|_1\leq C\int g(z_2,z_4)dz_2dz_4=C(\caC_n\ast\caC_n^2\ast \Ga_n^N)(0)<\infty
$$
since $ \Ga_n^N\in L^1$ and $\caC_n\ast\caC_n^2\in L^\infty$.

\vskip 2mm

Now we turn to the bi-local fields $\zeta(z_1,z_2)$ and set
\beq
\bbE\zeta(z_1,z_2)\zeta(z_3,z_4):=
H_n(\mathbf z)%(z_1,z_2,z_3,z_4)
\non%  \label{eq: kxiesti}
\eeq
We proceed as in \eqref{eq: ronbound}
\baq
\bbE\rho_n(z',z)^2&\leq& C \int_{(\bbR\times\bbT_n)^4}
  e^{-\hf( |z-z_1|+ |z-z_3|+ |z'-z_2|+ |z'-z_4|)}H_n(\mathbf z)%H_n(z_1,z_2,z_3,z_4)
  d\mathbf z. %\prod_idz_i.\non
  \label{eq: ronboundbi}
\eaq
Let
\beq
Y:=\sup_i\sum_j\sup_{z'\in c_i,z\in c_j} \bbE\rho_n(z',z)^2e^{c|z-z'|}.
\non%  \label{eq: kxiesti}
\eeq
\eqref{eq: twopoint2} follows from  $Y<\infty$. Let
\beq
\caH_n(\mathbf z):=\sup_{u}H_n(z_1+u,z_2+u,z_3+u,z_4+u)e^{\hf(|z_1-z_2|+|z_3-z_3|)}
\non%  \label{eq: kxiesti}
\eeq

We have
\baq
Y&\leq& C\sup_z \int_{(\bbR\times\bbT_n)^4}
  e^{-\hf( |z-z_1|+ |z-z_3|+ |z'-z_2|+ |z'-z_4|+|z_1-z_2|+|z_3-z_3|)}e^{c|z-z'|}\caH_n(\mathbf z)d\mathbf z.\non
 % \label{eq: ronbound}
\eaq
The integrand is actually independent on $z$ as $\caH_n$ is translation invariant:
\beq
\caH_n(\mathbf z)=\tilde\caH_n(z_{14}, z_{24}, z_{34})
\non%  \label{eq: kxiesti}
\eeq
with $\tilde\caH_n(z_1,z_2,z_3)=\caH(z_1,z_2,z_3,0)$. We can then conclude
\beq
Y\leq C\|\tilde\caH_n\|_1
\non%  \label{eq: kxiesti}
\eeq
i.e. we need to show for the various bi-local fields that $\|\tilde\caH_n\|_1<\infty$. Let us again proceed by cases.

\vskip 2mm

 \noindent (d)  $\zeta_n= :\eta_nC_n \Gamma_n^N\eta_n :
 $. We have  
 \beq
e^{\hf(|z_{12}|+|z_{34}|)}\caH_n(\mathbf z))\leq \tilde\caS_n(z_{12})\tilde\caS_n(z_{34})(
\caC_n(z_{13})\caC_n(z_{24})+\caC_n(z_{14})\caC_n(z_{23}))\non
\eeq
where $\tilde\caS_n(z)=e^{c|z|}\caS_n(z)$ so that 
 \beq
\|\tilde\caH_n\|_1\leq 2\|\tilde\caS_n\ast\caC_n\|_2^2<\infty
\non%  \label{eq: kxiesti}
\eeq
 since by Lemma \ref{lem: greg} $\tilde\caS_n$ is in $L^p$, $p<5/4$ and $\caC_n$ is in $L^p$, $p<5$
 so that $\tilde\caS_n\ast\caC_n\in L^p$, $p<\infty$. \eqref{eq: twopoint3} goes in the same way where at least one of the $\caC_n$ or $\caS_n$ is replaced by $\caC_n^{(N)}$ or $\caS_n^{(N)}$.

\vskip 2mm

 \noindent (e)  $\zeta_n= :\eta^2_n \Gamma_n^N\eta^2_n :$. We get
\baq
e^{\hf(|z_{12}|+|z_{34}|)}\caH_n(\mathbf z))&\leq& 4\tilde\Gamma_n^N(z_{12})\tilde\Gamma_n^N(z_{23})(
\caC_n(z_{13})^2\caC_n(z_{24})^2+\caC_n(z_{14})^2\caC_n(z_{23})^2\non\\
&+&\caC_n(z_{13})\caC_n(z_{24})\caC_n(z_{14})\caC_n(z_{23}))\non
\eaq
 The first two terms on the RHS have the same topology as in (d): their contribution to $\|\tilde\caH_n\|_1$
 is bounded by  $C\|\tilde\Gamma_n^N\ast\caC^2_n\|_2^2$ which is finite since $\tilde\Gamma_n^N$ is in $L^1$ and $\caC^2_n$ in $L^2$.
 
 The third term is treated as the analogous one $\caD$ in (c), let us call it  $\caD$ again.
 Again  its $L^1$ norm is given by
 $\|\caD\|_1=\int f(z_2,z_3)g(z_2,z_3)dz_2dz_3$ where
  $f$ is as above and is in $L^\infty$.  
  $g(z_2,z_3)=\caC_n(z_2-z_3)\caC_n(z_2)\Ga_n^N(z_3)$and thus
  $$
 \|\caD\|_1\leq C\int g(z_2,z_4)dz_2dz_4=C(\caC_n\ast\caC_n\ast \Ga_n^N)(0)<\infty.
$$

Let us finally turn to the $\delta_n$ terms in $\omega_n$,  $\zeta_n$
and  $z_n$. Starting with $\omega_n$, and the term $\zeta_n=\delta_n\Gamma^N_n:\eta_n^{3}:$ we have
$$
H_n(z_1,z_2)\leq \delta_n(t_1)\delta_n(t_2)(\Gamma^N_n\ast\caC_n^3\ast\Gamma^N_n)(z_{12}).
$$
The function $g_n=\Gamma^N_n\ast\caC_n^3\ast\Gamma^N_n$ is in $L^p$, $p<5/3$. Letting
$f_n(z_i)=e^{-c|z-z_i|} \delta_n(t_i)$ we see from Lemma \ref{lem: heatkernel}  that $f_n\in L^p$, $p<2$
and so
\beq
\bbE\rho_n(z)^2\leq C(f_n\ast g_n\ast f_n)(0)\non
 % \label{eq: ronbound}
\eeq
is finite by Young. To get the bound  \eqref{eq: twopoint3} we use \eqref{eq: deltaNnbound} to get for $p<2$
$$
\|f'_n%^{(N)}
-f_n%^{(N')}
\|_p\leq C\la^{\frac{2-p}{2p}(N-n)}\|\chi-\chi'\|_\infty.
$$
Consider next the case where $\delta_n$ is on the "other side": $\zeta_n=:\eta_n^{2}:\Gamma^N_n\delta_n\eta_n$. Replacing $C_n$ by the translation invariant upper bound $\caC_n$ we have
$$
H_n(z_1,z_2)\leq \int G_n(z_1,z_2,z_3,z_4)\delta_n(t_3)\delta_n(t_4)dz_3dz_4$$
with 
\beq
G_n(\mathbf z)=\Gamma^N_n(z_{13})\Gamma^N_n(z_{24})(2\caC_n(z_{12})^2\caC_n(z_{34})+4\caC_n(z_{12})\caC_n(z_{13})\caC_n(z_{24})+4\caC_n(z_{12})\caC_n(z_{14})\caC_n(z_{23}))
\non
  %\label{eq: ronbound}
\eeq
At the cost of replacing $\caC_n$ and $\Gamma^N_n$ by $\tilde\caC_n$ and $\tilde\Gamma^N_n$ we 
may replace $\delta_n$ by the $f$ in the upper bound for $\bbE\rho_n(z)^2$:
\beq
\bbE\rho_n(z)^2\leq  \int \tilde G_n(z_1,z_2,z_3,z_4)f(z_3)f(z_4)dz_1dz_2dz_3dz_4=
 \int g_n(z_3-z_4)f(z_3)f(z_4)dz_3dz_4\non
  %\label{eq: ronbound}
\eeq
This is bounded if $g_n$ is in $L^p$, $p>1$ which is now straightforward.

As the last case consider an example of a bi-local field the second last  term in $z_n$: $\zeta_n=\delta_n\Gamma^N_n:\eta_n^2:$. By the now familiar steps
\beq
\bbE\rho_n(z)^2\leq  \int g(z_1-z_2)f(z_1)f(z_2)dz_1dz_2
 \non
  %\label{eq: ronbound}
\eeq
with $g=  \tilde\Gamma^N_n\ast \tilde\caC^2\ast  \tilde\Gamma^N_n $ in $L^p$, $p<5/2$.

\qed
%:proof of prop 6.1. 
\subsection{Renormalization}

We are left with the first terms in \eqref{eq: tildeomega} and \eqref{eq: tildez} that require fixing the renormalization constants $m_2$ and $m_3$. Define (product is of kernels as usual)
 \beq
B^{(N)}_n=
(C^{(N)}_n)^2\Ga^{N}_n
 \label{eq:BNndeff}
\eeq

Let $\caB_\infty$ denote the set of bounded operators $L^\infty(\bbR\times\bbT_n)\to L^\infty(\bbR\times\bbR^3)$ and  $\caB_1$ the ones  $L^{\infty,1}(\bbR\times \bbT_n)\to L^\infty(\bbR\times\bbR^3)$ (here $L^{\infty,1}$ has 
$L^\infty$-norm in $t$ and $L^1$-norm in $x$).
We prove:

\begin{proposition}\label{pro: renormalization} There exist  
 constants $\beta_i$ s.t. the operator 
$$
R^{(N)}_n:=\tilde K(B^{(N)}_n-(\beta_2\log\la^{N-n}+\beta_3) {\ 
 id})
 $$ 
 satisfies
 \beq
\|R^{(N)}_n\|_{\caB_i}\leq C,\ \ \  \|{R'}^{(N)}_n-R^{(N)}_n\|_{\caB_i}\leq C\la^{ (N-n)}\|\chi-\chi'\|_\infty,\ \ \ i=1,\infty
\label{eq:MNbound1 }
\eeq
uniformly in $0\leq n\leq N<\infty$.  $\beta_2$ universal (i.e. independent on $\chi$).
 
\end{proposition} 

We fix the renormalization constants $m_i=18\beta_i$. This means that e.g. the first term in
\eqref{eq: tildeomega} becomes 
$$
18R^{(N)}_n\eta_n+(m_2\log\la^n+m_3)\eta_n.
$$
The second term once multiplied by $\la^{2n}$ fits into the bound \eqref{eq:Unormbound}.

The $i=\infty$ case of Proposition \ref{pro: renormalization}  takes care of the deterministic term in \eqref{eq: tildez} since
$$
\|h_{n-m}(B^{(N)}_n-\beta_1\log\la^{N-n}+\beta_2)
\phi\|_{\caV_n}\leq \|R^{(N)}_n\|_{\caB_\infty}\|\phi\|_\infty.
$$
and $\|\phi\|_\infty\leq \|\phi\|_{\Phi_n}$. The  $i=1$ case is needed for the first term in \eqref{eq: tildeomega}. Indeed,
we have
\beq
\bbE\rho^{(N)}_n(z)^2= \int (R^{(N)}_n(t,t_1)\ast R^{(N)}_n(t,t_2)\ast C^{(N)}_n(t_1,t_2))(0)dt_1dt_2
%  \label{eq: twopoint5}
\non
\eeq
where  $\ast$ is spatial convolution. We have
\beq
 \sup_{t_1,t_2}\|C^{(N)}_n(t_1,t_2)\|_1<\infty%\label{eq: }
\non
\eeq
so that indeed
\beq
\bbE\rho^{(N)}_n(z)^2\leq C \|R^{(N)}_n\|_{\caB_\infty}\|  \|R^{(N)}_n\|_{\caB_1}.%  \label{eq: twopoint5}
\non
\eeq
Recall also that by \eqref{eq: ronbound} a sufficient condition for $\|\tilde KB\|_{\caB_i}$ to be bounded is
 \beq
 %\caB:=
 \sup_t|B(t+\cdot,t,\cdot)|\in L^1(\bbR\times \bbT_n).
 \label{eq: suffi}
\eeq
%\eqref{eq:MNbound1 } to hold is 

\begin{proof} Let us first
remark that it suffices to work in $\bbR^3$ instead of $\bbT_n$. Indeed,  recall  $B^{(N)}_n=18
(C^{(N)}_n)^2B^{(N)}_n$ where the product is defined as pointwise multiplication of kernels.
Let $\tilde C^{(N)}_n$ and $\tilde \Gamma^{(N)}_n$ be given by \eqref{eq:Cndef } and \eqref{eq: gamman-ndef} where the $\bbT_n$ heat kernels $H_n$ are replaced by the  $\bbR^3$ heat kernel $H$ and similarly let
 Let
 \beq
 \delta B^{(N)}_n(t',t,x):= B^{(N)}_n(t',t,x)-\tilde B^{(N)}_n(t',t,x)
\non
\eeq
 From \eqref{eq:persum } and a similar representation for $\Gamma^{(N)}_n$ 
 we infer that $ C^{(N)}_n-\tilde C^{(N)}_n$ and  $ \Gamma^{(N)}_n-\tilde\Gamma^{(N)}_n$ are in $L^p(\bbR\times\bbT_n)$ for all
 $1\leq p\le\infty$. Proceeding as in Lemma \ref{lem: greg}  
 we then conclude  $\sup_t| \delta B^{(N)}_n(t+\tau,t,x) %B^{(N)}_n(t+\tau,t,x)- \tilde B^{(N)}_n(t+\tau,t,x)
 |$ is in $L^p$ for $p<5/4$ and the analogue of \eqref{eq: casnbou} holds. 
 Hence \eqref{eq: suffi} holds. 
 
For the rest of this Section we work in $\bbR\times\bbR^3$ and fix the UV cutoff $\la^{N-n}:=\ep$ and denote the operators simply by
$C$ and $\Gamma$. Also, we set $\chi_\ep(s):=\chi(s)-\chi(s/\ep^2)$. With these preliminaries 
we will start to work towards  extracting from the operator $B^{(N)}_n$ \eqref{eq:Bndef } the divergent part responsible for the renormalization. First we'll derive  a version of the  fluctuation-dissipation relation relating $C$ and $\Gamma$:
\beq
\partial_{t'}C(t',t)=-\hf \Gamma_{}(t',t)+A(t',t).\label{eq: fldiss}
\eeq
where $A(t',t)$ will give a non singular contribution to $B$.
To derive \eqref{eq: fldiss} write  \eqref{eq:Cndef } in operator form and differentiate  in $t'$:
\beq
\partial_{t'}C(t',t)=%e^{(t'+t)\Delta}\chi_{N-n}(t')\chi_{N-n}(t)+
e^{(t'-t)\Delta}\int_0^t(\Delta+\partial_{t'})e^{2s\Delta}\chi_{\ep}(t'-t+s)\chi_{\ep}(s)%(\chi(s)-\chi(\la^{-2n}s))^2
ds
\non
\eeq
Next, write $\Delta e^{2s\Delta}=\hf\partial_se^{2s\Delta}$ and integrate by parts to get 
\baq
\partial_{t}C(t',t)&=&-\hf e^{(t'+t)\Delta}\chi_{\ep}(t')\chi_{\ep}(t)+
e^{(t'-t)\Delta}\int_0^te^{2s\Delta}(\partial_{t}-\hf\partial_s)(\chi_{\ep}(t'-t+s)\chi_{\ep}(s))
ds\non\\
&:=&a_0(t',t)+\tilde a(t',t)\non
\eaq
Now recall that $\chi_{\ep}(s)=\chi(s)-\chi(\ep^{-2}s)$ and write denoting $\tau=t'-t$:
\beq
%&&(\partial_{t}+\hf\partial_s)(\chi_{N-n}(t'-t+s)\chi_{N-n}(s))=\hf(-\chi'_{N-n}(t'-t+s)\chi_{N-n}(s)+\chi_{N-n}(t'-t+s)\chi'(s))\non\\
%&&-\hf\chi_{N-n}(t'-t+s)\partial_s\chi(\la^{-2(N-n)}s):=\rho_{N-n}(t'-t,s)-\hf\chi_{N-n}(t'-t+s)\partial_s\chi(\la^{-2(N-n)}s)\non
(\partial_{t'}-\hf\partial_s)(\chi_{\ep}(t'-t+s)\chi_{\ep}(s))=\rho_1(\tau,s)+\rho_2(\tau,s)+\rho_3(\tau,s)
\eeq
with
\baq
\rho_1(\tau,s)&=&\hf(-\chi_{\ep}(\tau+s)\chi'(s)+\chi'(\tau+s)\chi_{\ep}(s))
\non\\
\rho_2(\tau,s)&=&-\hf\chi_{\ep}(s)\partial_s\chi(\ep^{-2}(\tau+s))
\non\\
\rho_3(\tau,s)&=&\hf\chi_{\ep}(\tau+s)\partial_s\chi(\ep^{-2}s)
\non
\eaq
and correspondingly 
$$\tilde a(t',t)=a_1(t',t)+a_2(t',t)+a_3(t',t).$$
$\rho_1$ localizes $s$-integral to $s>\caO(1)$ and gives a smooth contribution as $\ep\to 0$.
$d\mu(s):=-\partial_s\chi(\ep^{-2}s)ds$ is a
probability measure supported on $[\ep^{2},2\ep^{2}]$ so $\rho_2$ localizes $s$ and $\tau$
to $\caO(\ep^{2})$. % and will be bounded  as $\ep\to 0$. 
 The main term comes from $\rho_3$
\baq
%\partial_{t}C(t',t)=
a_3(t',t)&=&-\hf
e^{\tau\Delta}\int_0^te^{2s\Delta}\chi_{\ep}(\tau+s)d\mu(s)
\non\\
&=&-\hf \Gamma(t',t)+a_4(t',t)
\eaq
where
\beq
a_4(t',t)=\hf
e^{\tau\Delta}\int_0^t(\chi_{\ep}(\tau)-e^{2s\Delta}\chi_{\ep}(\tau+s))d\mu(s)%\tilde A_n(t',t)=%\hf e^{(t'+t)\Delta}\chi_{N-n}(t')\chi_{N-n}(t)
%e^{(t'-t)\Delta}\int_0^te^{2s\Delta}\rho_n(t'-t,s)
\label{eq:a4definition }%\non
\eeq
\eqref{eq: fldiss} follows then with
\beq
A(t',t)=a_0(t',t)+a_1(t',t)+a_2(t',t)+a_4(t',t).\label{eq:Andefinition }
%\non
\eeq
Inserting  \eqref{eq: fldiss} into 
 \eqref{eq:Bndef } and using
 \beq
 \int dy\int_0^tds\partial_tC(t,s,x-y)^{ 3} \phi(s,y)= \partial_t\int dy\int_0^tdsC(t,s,x-y)^{ 3} \phi(s,y)- \int C(t,t,x-y)^{ 3}\phi(t,y)dy
%\label{eq: }
\non
\eeq
we obtain
\beq
(B\phi)(t)=D(t)\ast \phi(t)% \int C(t,t,x-y)^{ 3}\phi(t,y)dy
+\partial_t(E\phi)(t)%)(t,x).%\non
+(F\phi)(t)
\label{eq: bb'deco}
\eeq 
with
\baq
D(t,x)&=&12C(t,t,x)^{ 3}
%\label{eq: }
\non\\
(E\phi)(t,x)&=&-12 \partial_t\int dy\int_0^tdsC(t,s,x-y)^{ 3} \phi(s,y)\non\\
(F\phi)(t,x)&=&-36 \int dy\int_0^tdsA(t,s,x-y)C(t,s,x-y)^{ 2} \phi(s,y).\non
\eaq
%\beq(B'\phi)(t,x)=-12 \partial_t\int dy\int_0^tdsC(t,s,x-y)^{ 3} \phi(s,y)-36 \int dy\int_0^tdsA(t,s,x-y)C(t,s,x-y)^{ 2} \phi(s,y).\non\eeq 
We estimate these three operators in turn. 
%:D estimate
\vskip 2mm
%:(a)

\noindent (a) $D$. %We start with the  convolution operator in \eqref{eq: bb'deco}. 
%By the relation \eqref{eq:persum } we may work on $x\in\bbR^3$.
The Fourier transform of $D$ in $x$ is given by
\beq
\hat D(t,p)=\int \hat C(t,t,p+k) \hat C(t,t,k+q) \hat C(t,t,q)dkdq 
\label{eq:ftofD }
%\non
\eeq
where %$\int_ndkdq$ denotes the Riemann sum in $k,q\in (2\pi \epsilon_n\bbZ)^3$ and
\beq
 \hat C(t,t,q)=\int_0^te^{-sq^2}\chi_\ep(s)^2ds%-\chi(%\la^{-2(N-n)}
 %s/\ep^2))^2ds%,\ \ \ \la^{N-n}\equiv\epsilon%:=c_n(t,q,\chi)
% \leq \int_{\la^{2n}}^2e^{-sq^2}ds=q^{-2}(e^{-\la^{2n}q^2}-e^{-2q^2}).
\label{eq: cnestim1}
%\non
\eeq
%and in this section we denote for simplicity $\la^{N-n}\equiv\epsilon$.
%Using  $1_{[2%\ep^{2}
%\ep^2, 1]}(s)\leq \chi(s)-\chi(%\la^{-2(N-n)}
%s\ep^2)\leq 1_{[\ep^2%\ep^{2}
%, 2]}(s)$
%\beq
 %q^{-2}(e^{-2\ep^{2}q^2}-e^{-t\wedge 1q^2})1_{t\geq 2\ep^{2}}\leq  \hat C_n(t,t,q)%c_n(t,q,\chi)
% \leq \int_{\la^{2n}}^2e^{-sq^2}ds=
%\leq q^{-2}(e^{-\ep^{2}q^2}-e^{-t\wedge 2q^2})1_{t\geq \ep^{2}}.
%\label{eq: cnestim2}
%\non
%\eeq
The  integral  in \eqref{eq:ftofD } diverges at $p=0$ logarithmically  as
$\ep\to 0$. Doing the gaussian integrals over $k,q$ we have 
\beq
\hat D(t,p)=(4\pi)^{-3}\int_{[0,t]^3}e^{-\al(\bss) p^2}%d(s_1,s_2,s_3)
d(\bss)^{-3/2}\prod_{i=1}^3\chi_\ep(s_i)^2ds_i
%\label{eq: }
\non
\eeq
where $\al(\bss):=-\frac{s_1s_2s_3}{d(\bss)}$ and $d(\bss):=s_1s_2+s_1s_3+s_2s_3$. Let us
study the cutoff dependence of $\hat D$. First by
 differentiating and changing variables
\beq
\ep\partial_\ep\hat D(t,p)=-\frac{3}{32\pi^3}\int_{[0,t/\ep^2]^3}
e^{-\al(\bss)(\ep p)^2}d(\bss)^{-3/2}s_1\partial_{s_1}\chi(s_1)^2ds_1
\prod_{i=2}^3(\chi(\ep^2s_i)^2-\chi(s_i)^2)ds_i.
\label{eq: epdep}
%\non
\eeq
Let
\beq
\al_\ep(t,p)
:=-\frac{3}{32\pi^3}\int_{[0,t/\ep^2]^3}
e^{-\al(\bss)(\ep p)^2}d(\bss)^{-3/2}s_1\partial_{s_1}\chi(s_1)^2ds_1
\prod_{i=2}^3(1-\chi(s_i)^2)ds_i.
\label{eq: epdep}
%\non
\eeq 
%$\al_\ep(t,p)$ be given by \eqref{eq: epdep} where $\chi(\ep^2s_i)^2$ is replaced by $1$ 
and set
%\beq
$\tilde\al_\ep(t,p)=\ep\partial_\ep\hat D(t,p)-\al_\ep(t,p)$.
%\label{eq: epdep}
%\non
%\eeq
Since the $s_1$ integral is supported on $[1,2]$ and the others on $s_i\geq 1$ we have
$
d(\bss)\asymp s_2s_3,\ \ \ \al(\bss)\asymp 1
$
which leads to
\beq
|\tilde\al_\ep(t,p)|\leq C\int_{\bbR_+^3}(s_2s_3)^{-3/2}1_{[1,2]}(s_1)1_{[\ep^{-2},\infty)}(s_2)
1_{[1,\infty)}(s_3)\leq C\ep.
\label{eq: epdep1a}
%\non
\eeq
%since $1-\chi(\ep^2s_i)^2$  is supported on $[\ep^{-2},\infty]$.
Furthermore, let $\tilde\al'_\ep(t,p)$ be gotten by replacing the lower cutoffs $\chi(s_i)$
by  another one $\chi'(s_i)$. Then
\beq
|\tilde\al'_\ep(t,p)-\tilde\al_\ep(t,p)|\leq C\ep\|\chi-\chi'\|_\infty.
\label{eq: epdep1b}
%\non
\eeq
Since $\tilde\al_0(t,p)=0$ we get that $\int_0^\ep\tilde\al_{\ep'}(t,p)\frac{d\ep'}{\ep'}$ satifies 
\eqref{eq: epdep1a} and \eqref{eq: epdep1b} as well.
Thus all the divergences come from $\al_\ep(t,p)$. Note that $\al_0(t,p)=\al_0$ is independent on $t$ and $p$. Set
 $a_\ep(t,p)=\al_\ep(t,p)-\al_0$.
We get
\beq
|a_\ep(t,p)|\leq C((\ep^2p^2+\ep/\sqrt{t})\wedge 1)
\label{eq: epdep1aa}\\
%\non
%|a'_\ep(t,p)-a_\ep(t,p)|&\leq &C((\ep^2p^2+\ep/\sqrt{t})\wedge 1)\|\chi-\chi'\|_\infty
%\label{eq: epdep1aaa}
%\non
\eeq
We fix the renormalization constant $\beta_2=\al_0$ and define
\beq
d(t,p):=\hat D(t,p)-\al_0\log\ep.
\label{eq: smallddef}
%\non
\eeq
Combining above we get
\beq
|d(t,p)|=|\int_\ep^1(\tilde\al_{\ep'}(t,p)+a_{\ep'}(t,p))\frac{d\ep'}{\ep'}|\leq C(1+\log(1+p^2+1/\sqrt{t})).
\label{eq: smallddef}
%\non
\eeq
Next we write
\baq
\al_\ep(t,p)&=&-\frac{1}{32\pi^3}\int_{[0,t/\ep^2]^3}
e^{-\al(\bss)(\ep p)^2}d(\bss)^{-3/2}\sum_is_i\partial_{s_i}
\prod_{i=1}^3(1-\chi(s_i)^2)ds_i
%\label{eq: }
\label{eq: ipp}\\
&=&-\frac{1}{32\pi^3}\int_0^\infty d\la
\int_{[0,t/\la\ep^2]^3}\delta(\sum_{i=1}^3s_i-1)e^{-\la\al(\bss)(\ep p)^2}d(\bss)^{-3/2}
\partial_\la
\prod_{i=1}^3(1-\chi(\la s_i)^2)ds_i
\non
\eaq
At $\ep=0$ this implies after an integration by parts
\beq
\al_0
=
-\frac{1}{32\pi^3}\int_{\bbR_+^3}\delta(\sum_{i=1}^3s_i-1)d(\bss)^{-3/2}%(s_1s_2+s_1s_3+s_2s_3)^{-3/2}
\prod_{i=1}^3ds_i.
\non
\eeq
i.e.  $\al_0$ is {\it universal} (i.e. independent of  $\chi$). Finally let us vary the cutoff. Replace 
$\chi$ by $\chi_\sigma=\sigma\chi+(1-\sigma)\chi'$. Since $\partial_\sigma a_\ep=\partial_\sigma (a_\ep-a_0)$
we get from \eqref{eq: ipp} 
\baq
\partial_\sigma\al_\ep(t,p)
&=&%-\frac{1}{32\pi^3}
\int_0^\infty d\la
\int_{[0,t/\la\ep^2]^3}%\delta(\sum_{i=1}^3s_i-1)
\al(\bss)(\ep p)^2e^{-\la\al(\bss)(\ep p)^2}d(\bss)^{-3/2}
d\mu(\bss)\non\\
&+&
\sum_{i=1}^3\int_0^\infty d\la(t/\la^2\ep^2)
\int_{[0,t/\la\ep^2]^2}%\delta(\sum_{i=1}^3s_i-1)
e^{-\la\al(\bss)(\ep p)^2}d(\bss)^{-3/2}
d\mu(\bss)|_{s_i=t/\la\ep^2}
%\prod_{i=1}^3\chi(\la s_i)^2ds_i%\non
\label{eq: ipp1}
\eaq
where 
$$
d\mu(\bss)=\frac{3}{16\pi^3}\delta(\sum_{i=1}^3s_i-1)(\chi_\sigma(\la s_1)-\chi'_\sigma(\la s_1))
\chi_\sigma(\la s_1)\prod_{i=2}^3(1-\chi_\sigma(\la s_i)^2).
$$
Start with the first term in \eqref{eq: ipp1}.  Since $\la s_i\geq 1$ the $\la$-integral is supported in
$\la\geq 3$. On the support of $\chi_\sigma(\la s_1)-\chi'_\sigma(\la s_1)$
$s_1\in [\la^{-1},2\la^{-1}]$. By symmetry we may assume $s_2\leq s_3$ and then $s_3\geq 1/6$.
Hence in the support of the $\chi$'s $\al\asymp s_1$ and $d\asymp s_2$. We get the bound
$$
C(\ep p)^2e^{-(\ep p)^2}\int_{1/3}^\infty d\la\int_{1/\la}^{2/\la}ds_1\int_{1/\la}^{1}ds_2s_2^{-3/2}\ep\|\chi-\chi'\|_\infty\leq C(\ep p)^2e^{-(\ep p)^2}\|\chi-\chi'\|_\infty.
$$
For the second term, if $i=1$ then $s_1=t/\la\ep^2\in [1/\la,2/\la]$ implies $t\in [\ep^2,2\ep^2]$.
Again $\al\asymp s_1$, $d\asymp s_2\leq s_3$ and we end up with the bound
$$
C1(t\in [\ep^2,2\ep^2])\|\chi-\chi'\|_\infty.
$$
If $i\neq 1$ the same bound results. We may summarize this discussion in
\beq
|d(t,p)-d'(t,p)|\leq C(\ep+1(t\in [\ep^2,2\ep^2])\|\chi-\chi'\|_\infty.
\label{eq: smallddef}
%\non
\eeq
	The operator $K(t',t)d(t)$ acts as a Fourier multiplier with $ \hf e^{-|t'-t|}(p^2+1)^{-2}d(t,p)$.
Since $d(t,p)$ is analytic in a strip $|\Im p|\leq c$ we get in $x$-space from the
above bounds
\baq
|(K(t',t)d(t))(x)|&\leq& Ce^{-|t'-t|-c|x|}(1+\log(1+t^{-\hf}))
%\log(2+(\ep^2+t)^{-1})%(1+ 1_{t\in[\ep^2,1]}\log 1/t)
\non\\
|(K(t',t)(d'(t)-d(t)))(x)|&\leq &Ce^{-|t'-t|-c|x|}(\ep +1(t\in [\ep^2,2\ep^2]))\|\chi-\chi'\|_\infty.
\non
\eaq
Hence
\beq
|(\tilde Kd\phi)(t',x)|\leq \int e^{-\hf|t'-t|-c|x-y|}(1+\log(1+t^{-\hf}))\phi(t,y)dtdy
\non
\eeq
which is in $L^\infty$ for $\phi\in L^\infty\times L^\infty$ and for $\phi\in L^\infty\times L^1$ as well. This gives the first bound in  \eqref{eq:MNbound1 }. The second is similar.
%$$|(K(d-d')1_{I_n}\phi)(t,x)|\leq  C\ep e^{-\dist(t,I_n)}\|\phi\|_\infty$$
%which implies the claim of the Proposition.
\vskip 2mm
%:(b)
\noindent (b) $E$. We have $E\phi=\partial_tC^3\phi$ where $C^3(t,s,x):=C(t,s,x)^3$. Thus
integrating by parts %So recalling \eqref{eq: Kdef} and that $K_1$ has kernel $\hf e^{-|t'-t]}$ we get
\beq
Kh_{n-m}E\phi=K\partial_th_{n-m}C^3\phi+\partial_tKh_{n-m}C^3\phi.
\label{eq: deccom}
%\non
\eeq
By  \eqref{eq: K_1def} $|\partial_tK(z)|=K(z)$ so we may use   \eqref{eq: Ktildenewbound}
for the second term as well to get
$$
\|\tilde Kh_{n-m}E\phi\|_\infty\leq C\|\caK\ast\caC^3\ast\phi\|_\infty.
$$
By Lemma \ref{lem: greg}  $\caC^3$ is in $L^1(\bbR\times\bbT_n)$ and hence
$\caK\ast\caC^3$ is in $L^1(\bbR\times\bbT_n)$ and in $L^1(\bbR)\times L^\infty(\bbT_n)$
so that the first estimate of  \eqref{eq:MNbound1 } follows. The second is similar.

\vskip 2mm

%:(c)
\noindent (c) $F$.   By \eqref{eq:Andefinition } $F$  has four contributions, call them
$F_0,F_1,F_2, F_4$. Start with $F_1$. Since $\rho_1$ is supported in $s\geq 1$ the kernel
is bounded (in fact smooth)
$$
|a_1(t',t,x)|\leq Ce^{-c|x]}.
$$
and so by \eqref{eq: cbound} we get
$$
|F_1(t+\tau,t,x)|\leq Ce^{-c|x]}(x^2+\tau)^{-1}.
$$
Hence \eqref{eq: suffi} holds for $F_1$
uniformly in $\ep$ and the first bound in  \eqref{eq:MNbound1 } follows. For the second one we proceed as in Lemma \ref{lem: greg} to get for  $f_1(\tau,x):=\sup_t|F_1(t+\tau,t,x)-F'_1(t+\tau,t,x)|$ that 
%\beq
$\|f_1\|_1%b_1^{(N')}-b_1^{(N)}\|_p^p
\leq C\ep^{\ga}$
for some $\ga>0$.%\la^{(5/2-p)(N-n)}
%\label{eq: boneest}
%\non
%\eeq
%Let $\tilde K(t,x)=e^{\hf|t|}K(t,x)$. Then
%\beq
%|(K(F_1-F'_1)1_{I_n}\phi)(t,x)|\leq Ce^{-\hf |t-\tau_n|}\|\tilde K\ast b_1\ast \phi\|_\infty
%\leq Ce^{-\hf |t-\tau_n|}\|\tilde K\|_1\|b_1\|_1\|\phi\|_\infty
%\label{eq: }
%\non
%\eeq
%which combined with \eqref{eq: boneest} gives the claim.

 Consider next $F_2=-36a_2C^2$.  We show:
 \beq
F_2(t',t,x)=\ep^{-5}p((t'-t)/\ep^2, x/\ep)
%\frac{\tau}{\ep^2},\frac{x}{\ep})
+r(t',t,x)
\label{eq: Ftwo2deco}
%\non
\eeq
where 
\beq
|p(\tau,x)|\leq Ce^{-cx^2}1_{\tau\leq 2}%t\wedge 2.
\label{eq: pikkup}
%\non
\eeq
and $r$ satisfies  \eqref{eq:MNbound1 }. To derive \eqref{eq: Ftwo2deco} we note that by
a change of variables
 %\eqref{eq:cMtaux }
 \beq
a_2(t+\tau,t,x)=-\hf \int_0^{t/\ep^2}H(\tau+2\ep^2s,x)(1-\chi(s))\partial_s\chi(\tau/\ep^2+s)ds\non%\\&=&-\hf \ep^{-3}\int_0^{t/\ep^2}H_N(\tau/\ep^2+s,x/\ep)(1-\chi(s))\chi'(\tau/\ep^2+s)ds:= \ep^{-3}\alpha_N(\frac{\tau}{\ep^2},\frac{x}{\ep},\frac{t}{\ep^2})\non
%(4\pi (\tau+s))^{-3/2}e^{-\frac{x^2}{4(\tau+s)}}ds1_{\tau\in[0,2\ep^{2} ]}+C\ep^{2}e^{-c|x]}.
%\int H_n(\tau+2s,x)%(4\pi (\tau+2s))^{-3/2}e^{-\frac{x^2}{4(\tau+2s)}}
%d\mu(\tau+s)%+C\ep^{2}e^{-c|x]}
%\leq C\la^{-3(N-n)}e^{-c|x|/\la^{N-n}}1_{\tau\in[\ep^{2} ,2\ep^{2} ]} .\non
\eeq
where we also noted since $\partial_s\chi$ is supported on $[1,2]$  $\chi(\ep^2s)=1$. Using
 scaling property of the heat kernel $H(\tau+\ep^2s,x)=\ep^{-3}H(\tau/\ep^2+2s,x/\ep)$ we get then
\beq
a_2(t+\tau,t,x)=%-\hf \ep^{-3}\int_0^{t/\ep^2}H(\tau/\ep^2+s,x/\ep)(1-\chi(s))\chi'(\tau/\ep^2+s)ds:=
 \ep^{-3}
\alpha(\frac{\tau}{\ep^2},\frac{t}{\ep^2},\frac{x}{\ep})\non
\eeq
with 
\beq
\alpha(\tau,t,x)
=-\hf \int_0^{t}H(\tau+s,x)(1-\chi(s))\partial_s\chi(\tau+s)ds.\label{eq: alphantau}
\eeq
 Note that $\al$ depends on $t$ only on $t\leq 2$ and is bounded by %supported on $\tau\leq 2 $. We have
\beq
|\alpha(\tau,t,x)|\leq Ce^{-cx^2}1_{\tau\leq 2},\ \ \ |\alpha(\tau,t,x)-\alpha(\tau,\infty,x|\leq Ce^{-cx^2}1_{\tau\leq 2}1_{t\leq 2}.%t\wedge 2.
%\label{eq: }
\non
\eeq
Comparing two lower cutoffs $\chi$ and $\chi'$ we get 
\beq
|\alpha(\tau,t,x)-\alpha'(\tau,t,x|\leq Ce^{-cx^2}1_{\tau\leq 2}\|\chi-\chi'\|_\infty.%t\wedge 2.
%\label{eq: }
\non
\eeq
%We need to consider $\beta_2(t',t,x)=a_2(t',t,x)C^{(N)}(t',t,x)$ i.e. $C^{(N)}$ enters only 
%for $\tau=\caO(\ep^2), |x|=\caO(\ep)$. 
By similar manipulations we obtain
\beq
C(t+\tau,t,x)^2%=\ep ^{-1}\int_0^{t/\ep^2}H(\tau/\ep^2+s,x/\ep)(\chi(\ep^2(\tau+s))-\chi(\tau+s))(\chi(\ep^2s)-\chi(s))ds\non\\
=\ep^{-2}c(\frac{\tau}{\ep^2},\frac{t}{\ep^2}, \frac{x}{\ep};\ep)^2
%\label{eq: }
\non
\eeq
where
\beq
c(\tau,t,x;\ep)=\int_0^{t}H(\tau+2s,x)(\chi(\ep^2(\tau+s))-\chi(\tau+s))(\chi(\ep^2s)-\chi(s))ds.
%\label{eq: }
\non
\eeq
  %since only $\tau\leq 2$ need be considered and
Since $H(\tau+2s,x)\leq C(1+s)^{-3/2}$ on support of the integrand we get
 \beq
 |c(\tau,t,x;\ep)|\leq C(1+|x|)^{-1}
 %\label{eq: }
\non
\eeq
 and
\beq
|c(\tau,t,x;\ep)-c(\tau,\infty,x;0)|\leq C(\ep(1+\ep |x|)^{-1}+(1+|x|+t)^{-\hf})
%|c(\tau,x,t;\ep)-c(\tau,x,\infty;0)|\leq C(\ep+(1+t)^{-\hf}).
%\label{eq: }
\non
\eeq
with an extra $\|\chi-\chi'\|_\infty$ factor if we compare two lower cutoffs.
%Furthermore, comparing two lower cutoffs
%\beq
%|c(\tau,t,x;\ep)-c'(\tau,t,x;\ep)|\leq C\|\chi-\chi'\|_\infty.
%|c(\tau,x,t;\ep)-c(\tau,x,\infty;0)|\leq C(\ep+(1+t)^{-\hf}).
%\label{eq: }
%\non
%\eeq
\eqref{eq: Ftwo2deco} follows with
$$
p(\tau,x)=
-36c(\tau,x,\infty,0)^2\alpha(\tau,\infty, x).
$$ 
The error term satisfies
\beq
 |r(t+\tau,t,x)|\leq C\ep^{-5}(\ep +(1+t/\ep^2)^{-\hf}+1_{t\leq 2\ep^2})e^{-cx^2/\ep^2}1_{\tau\leq 2\ep^2}
\label{eq:rmiIinus r' }
%\non
\eeq
with an extra $\|\chi-\chi'\|_\infty$ factor if we compare two lower cutoffs. Hence
\beq
|(\tilde Kr\phi)(t',x')|\leq \int e^{-c|t'-\tau-t|+|x'-x|}|r(t+\tau,t,x-y)|\phi(t,y)|d\tau dx\leq C\ep\|\phi\|  %\label{eq: }
\non
\eeq
both in the norm  $L^\infty\times L^\infty$ and in $L^\infty\times L^1$. This and similar statement with   $\|\chi-\chi'\|_\infty$  gives \eqref{eq:MNbound1 }. 

%Call the first term in \eqref{eq: Ftwo2deco} 
Let  $\tilde p$ be the operator with the kernel $\ep^{-5}p((t'-t)/\ep^2, (x'-x)/\ep)-p_0\delta(z'-z)$
where  $p_0=\int p(z)dz$. Then
\beq
(Kh_{n-m}\tilde p)(z+v,z)=\int (K(v-u_\ep)h(z+u_\ep)-K(v)h_{n-m}(z))p(u)du.
\label{eq:khtildep }
%\non
\eeq
where $u_\ep=(\ep^2u_0,\ep\bsu)$. The bound \eqref{eq: pikkup}  then implies  $\|\tilde K\tilde p\phi\|_\infty\leq C\ep\|\phi\|$ in
both norms. A similar statement holds  with   $\|\chi-\chi'\|_\infty$. $p_0$ contributes the the renormalization constant $m_2$.
\vskip 2mm

The analysis of $F_4=-36a_4C^2$ parallels that of $F_2$ so we are brief:
 \beq
F_2(t',t,x)=\ep^{-5}q((t'-t)/\ep^2, x/\ep)
%\frac{\tau}{\ep^2},\frac{x}{\ep})
+s(t',t,x)
\label{eq: Ftwo2deco}
%\non
\eeq
where $s$ satisfies  \eqref{eq:MNbound1 } and
$$
q(\tau,x)=
-36c(\tau,x,\infty,0)^2a(\tau,x)
$$
with
$$
a(\tau,x)=
\int(H(\tau,x)(1-\chi(\tau))-H(\tau+2s,x)(1-\chi(\tau+s))\partial_s\chi(s)ds.
$$
Since $\partial_s\chi$ is supported on $[1,2]$  and $1-\chi$ on $\tau\geq 1$ we get $|a(\tau,x)|\leq 
(1+\tau)^{-5/2}e^{-cx^2/\tau}$. Since $c(\tau,x,\infty,0)\leq C(1+|x|)^{-1}$ we end up with
$$
|q(\tau,x|\leq C(1+\tau)^{-5/2}(1+|x|)^{-2}e^{-cx^2/\tau}.
$$
We may now proceed as in \eqref{eq:khtildep }.
\vskip 2mm
The analysis of the term $F_0$ proceeds along similar lines and is omitted.

\end{proof}

%:biblio

\end{document}